\documentclass[12pt]{amsart}
\usepackage{amsmath,amssymb,amsbsy,amsfonts,latexsym,amsopn,amstext,cite,amsxtra,euscript,amscd,bm,mathabx,mathrsfs, abraces}
\usepackage{url}
\usepackage[colorlinks,linkcolor=blue,anchorcolor=blue,citecolor=blue,backref=page]{hyperref}
\usepackage{color}
\usepackage{graphics,epsfig}
\usepackage{graphicx}
\usepackage{float} 
\usepackage[english]{babel}
\usepackage{mathtools}
\usepackage{todonotes}
\usepackage{url}
\usepackage[colorlinks,linkcolor=blue,anchorcolor=blue,citecolor=blue,backref=page]{hyperref}

\usepackage[norefs,nocites]{refcheck}

\hypersetup{breaklinks=true}

\def\ssum{\mathop{\sum\ldots \sum}}

\usepackage[norefs,nocites]{refcheck}
\usepackage[english]{babel}
\begin{document}

\newtheorem{theorem}{Theorem}
\newtheorem{thm}{Theorem}
\newtheorem{lem}[thm]{Lemma}
\newtheorem{claim}[thm]{Claim}
\newtheorem{cor}[thm]{Corollary}
\newtheorem{prop}[thm]{Proposition} 
\newtheorem{definition}[thm]{Definition}
\newtheorem{question}[thm]{Open Question}
\newtheorem{conj}[thm]{Conjecture}
\newtheorem{rem}[thm]{Remark}
\newtheorem{prob}{Problem}

%%%
%\def\ccr#1{\textcolor{red}{#1}}
%\def\ccm#1{\textcolor{magenta}{#1}}
%\def\cco#1{\textcolor{orange}{#1}}
%\def\ccg#1{\textcolor{cyan}{#1}}

\def\vA{\mathbf A}
\def\vB{\mathbf B}

\newtheorem{ass}[thm]{Assumption}

\def \ss {s}

\newtheorem{lemma}[thm]{Lemma}

\newcommand{\rad}{\operatorname{rad}}
\newcommand{\GL}{\operatorname{GL}}
\newcommand{\SL}{\operatorname{SL}}
\newcommand{\lcm}{\operatorname{lcm}}
\newcommand{\ord}{\operatorname{ord}}
\newcommand{\Tr}{\operatorname{Tr}}
\newcommand{\Span}{\operatorname{Span}}

\numberwithin{equation}{section}
\numberwithin{theorem}{section}
\numberwithin{thm}{section}
\numberwithin{table}{section}

\def\vol {{\mathrm{vol\,}}}
\def\squareforqed{\hbox{\rlap{$\sqcap$}$\sqcup$}}
\def\qed{\ifmmode\squareforqed\else{\unskip\nobreak\hfil
\penalty50\hskip1em\null\nobreak\hfil\squareforqed
\parfillskip=0pt\finalhyphendemerits=0\endgraf}\fi}

\newcommand{\rank}{\operatorname{rank}}

\def \balpha{\bm{\alpha}}
\def \bbeta{\bm{\beta}}
\def \bgamma{\bm{\gamma}}
\def \blambda{\bm{\lambda}}
\def \bchi{\bm{\chi}}
\def \bphi{\bm{\varphi}}
\def \bpsi{\bm{\psi}}
\def \bomega{\bm{\omega}}
\def \btheta{\bm{\vartheta}}
\def \bmu{\bm{\mu}}
\def \bnu{\bm{\nu}}

\newcommand{\bfxi}{{\boldsymbol{\xi}}}
\newcommand{\bfrho}{{\boldsymbol{\rho}}}

\def\cA{{\mathcal A}}
\def\cB{{\mathcal B}}
\def\cC{{\mathcal C}}
\def\cD{{\mathcal D}}
\def\cE{{\mathcal E}}
\def\cF{{\mathcal F}}
\def\cG{{\mathcal G}}
\def\cH{{\mathcal H}}
\def\cI{{\mathcal I}}
\def\cJ{{\mathcal J}}
\def\cK{{\mathcal K}}
\def\cL{{\mathcal L}}
\def\cM{{\mathcal M}}
\def\cN{{\mathcal N}}
\def\cO{{\mathcal O}}
\def\cP{{\mathcal P}}
\def\cQ{{\mathcal Q}}
\def\cR{{\mathcal R}}
\def\cS{{\mathcal S}}
\def\cT{{\mathcal T}}
\def\cU{{\mathcal U}}
\def\cV{{\mathcal V}}
\def\cW{{\mathcal W}}
\def\cX{{\mathcal X}}
\def\cY{{\mathcal Y}}
\def\cZ{{\mathcal Z}}
\def\Ker{{\mathrm{Ker}}}

\def\NmQR{N(m;Q,R)}
\def\VmQR{\cV(m;Q,R)}

\def\Xm{\cX_m}

\def \A {{\mathbb A}}
\def \B {{\mathbb A}}
\def \C {{\mathbb C}}
\def \F {{\mathbb F}}
\def \G {{\mathbb G}}
\def \L {{\mathbb L}}
\def \K {{\mathbb K}}
\def \N {{\mathbb N}}
\def \Q {{\mathbb Q}}
\def \R {{\mathbb R}}
\def \U{{\mathbb U}}
\def \Z {{\mathbb Z}}

\def\fA{{\mathfrak A}}
 \def\fB{{\mathfrak B}}
\def\fC{{\mathfrak C}}
\def\fD{{\mathfrak D}}
\def\fG{{\mathfrak G}}
\def\fH{{\mathfrak H}}
\def\fL{{\mathfrak L}}
\def\fM{{\mathfrak M}}
\def\fP{{\mathfrak P}}
\def\fQ{{\mathfrak P}}
\def\fS{{\mathfrak  S}}
\def\fU{{\mathfrak U}}
\def\fY{{\mathfrak Y}}

\def \fUg{{\mathfrak U}_{\mathrm{good}}}
\def \fUm{{\mathfrak U}_{\mathrm{med}}}
\def \fV{{\mathfrak V}}
\def \fG{\mathfrak G}

\def\e{{\mathbf{\,e}}}
\def\ep{{\mathbf{\,e}}_p}
\def\eq{{\mathbf{\,e}}_q}

 \def\\{\cr}
\def\({\left(}
\def\){\right)}
\def\fl#1{\left\lfloor#1\right\rfloor}
\def\rf#1{\left\lceil#1\right\rceil}

\def\Im{{\mathrm{Im}}}

\def \oF {\overline \F}

\newcommand{\pfrac}[2]{{\left(\frac{#1}{#2}\right)}}

\def \Prob{{\mathrm {}}}
\def\e{\mathbf{e}}
\def\ep{{\mathbf{\,e}}_p}
\def\epp{{\mathbf{\,e}}_{p^2}}
\def\em{{\mathbf{\,e}}_m}

\def\Res{\mathrm{Res}}
\def\Orb{\mathrm{Orb}}

\def\vec#1{\mathbf{#1}}
\def \va{\vec{a}}
\def \vb{\vec{b}}
\def \vc{\vec{c}}
\def \vs{\vec{s}}
\def \vu{\vec{u}}
\def \vv{\vec{v}}
\def \vw{\vec{w}}
\def\vlam{\vec{\lambda}}
\def\flp#1{{\left\langle#1\right\rangle}_p}

\def\mand{\qquad\mbox{and}\qquad}

\title[Statistics of rational  matrices]
{Statistics  of  ranks, determinants  and characteristic polynomials of rational  matrices}

\author[{\tiny M. Afifurrahman}] {Muhammad Afifurrahman}
\address{MA:  \  School of Mathematics and Statistics, University of New South Wales, Sydney NSW 2052, Australia}
\email{m.afifurrahman@unsw.edu.au}

\author[{\tiny V. Kuperberg}]{Vivian Kuperberg}
\address{VK: \ Department of Mathematics, Eidgen{\"o}ssische Technische Hoch\-schule Z{\"u}rich, 
Rämistrasse 101, 8092 Z{\"u}rich, Switzerland}
\email{vivian.kuperberg@math.ethz.ch}

\author[{\tiny A. Ostafe}] {Alina Ostafe}
\address{AO:  \  School of Mathematics and Statistics, University of New South Wales, Sydney NSW 2052, Australia}
\email{alina.ostafe@unsw.edu.au}

\author[{\tiny I. E. Shparlinski}] {Igor E. Shparlinski}
\address{IS:  \  School of Mathematics and Statistics, University of New South Wales, Sydney NSW 2052, Australia}
\email{igor.shparlinski@unsw.edu.au}

\begin{abstract}  We consider the set of $m\times n$ matrices with 
rational entries having numerator and denominator  of size at most $H$ and obtain various  upper bounds on the number of
such matrices of a given rank, or with a given determinant, or a given characteristic polynomial. 
We also consider similar questions for matrices whose entries are Egyptian fractions. 
\end{abstract}



\maketitle

\tableofcontents

\section{Introduction}

\subsection{Set-up} 
Various counting problems for matrices with integer elements have recently received a lot of attention. We refer to~\cite{ALPS, Afif, BOS, E-BLS, HOS, MeSh, MOS} for some 
recent results and further references. Here we consider an apparently new set-up
involving matrices with rational entries of restricted height.

More precisely, let 
\[
    \cF(H)=\{a/b:~a,b \in \Z, \ 0\le |a|, b \le H,\ \gcd(a,b)=1\}
\]
be the set of fractions in lowest terms of (naive)
height at most $H$. Clearly, $ \cF(H)$ is set of all \textit{Farey fractions} 
together with their inverses and  sign alternations. 

A particularly well-known counting result for Farey fractions implies that 
asymptotically the cardinality of $\cF(H)$ satisfies
\[
    \# \cF(H) \sim \frac{12}{\pi^2}H^2, \qquad \text{as}\ H \to \infty,
\]
where the best known error term is due to Walfisz~\cite[Chapter~V, Section~5, Equation~(12)]{Walf}. 

Similarly, we define the set of \textit{Egyptian fractions} 
\[
    \cE(H)=\{1/a:~ a \in \Z, \ 1\le |a|  \le H\},
\]
which are also known as {\it unit fractions\/}.
 Next, let 
\begin{align*}
& \cM_{m,n}\(\Q;H\)= \bigl\{A=\(a_{i,j}\)_{\substack{1\le i\le m \\ 1 \le j\le n}}\in \Q^{m\times n}:\\
& \qquad \qquad \qquad \qquad \qquad \qquad \quad ~ a_{i,j}\in \cF(H),\  1\le i\le m, \ 1 \le j\le n\bigr\},\\
& \cM_{m,n}\(\Z^{-1};H\)= \bigl\{A=\(a_{i,j}\)_{\substack{1\le i\le m \\ 1 \le j\le n}}\in \Q^{m\times n}:\\
& \qquad \qquad \qquad \qquad \qquad \qquad \quad ~ a_{i,j}\in \cE(H),\  1\le i\le m, \ 1 \le j\le n\bigr\}, \end{align*} 
be the sets of $m\times n$ 
 rational matrices with rational  entries from $ \cF(H)$ and $ \cE(H)$, respectively. In particular, asymptotically, we have
\[
\#\cM_{m,n}\(\Q;H\) \sim \( \frac{12}{\pi^2} H\)^{mn}\quad \text{and} \quad \#\cM_{m,n}(\Z^{-1},H) \sim  (2H)^{mn}  
\]
as $H \to \infty$.
Note that in the above and in all other asymptotic formulas, we always assume
that $m$ and $n$ are fixed (while $H \to \infty$). 

Here, we consider some new questions related to the arithmetic statistics of the matrices in $\cM_{m,n}\(\Q;H\)$ and $\cM_{m,n}(\Z^{-1},H)$. First, we study the distribution of ranks of  matrices in $\cM_{m,n}\(\Q;H\)$ and  
$\cM_{m,n}(\Z^{-1},H)$.  More precisely, we study  the cardinalities of the  sets 
\begin{align*}
&\cL_{m,n, r}\(\Q;H\)=\left \{A \in \cM_{m,n}\(\Q;H\):~\rank A = r \right \},\\
&\cL_{m,n, r}\(\Z^{-1};H\)=\left \{A \in \cM_{m,n}\(\Z^{-1};H\):~\rank A = r \right \}. 
\end{align*} 
For the case $m=n$ (the square matrices), we also consider the statistics of matrices with a given determinant  or characteristic polynomial. More precisely, we study  the cardinalities of the  sets
\begin{align*}
&    \cD_n(\Q; H,\delta)=\{A\in \cM_{n,n}\(\Q;H\):~ \det A=\delta \},\\
&    \cD_n(\Z^{-1}; H,\delta)=\{A\in \cM_{n,n}\(\Z^{-1};H\):~ \det A=\delta \}, 
\end{align*}  
for $\delta\in \Q$.
Also for $f \in \Q[X]$,  we study  the cardinalities of the  sets
\begin{align*}
    &\cP_n(\Q; H,f)=\{A\in \cM_n(\Q;H):\\
      &\qquad \qquad \qquad \qquad \qquad  \text{the characteristic polynomial of $A$ is $f$}\}\\
& \cP_n(\Z^{-1}; H,f)=\{A\in \cM_n(\Z^{-1};H):\\
    &\qquad   \qquad  \qquad \qquad \qquad  \text{the characteristic polynomial of $A$ is $f$}\}.
 \end{align*}

\subsection{Motivation and approach}

The above  problems can be seen as analogues of  the problems for the set $\cM_{m,n}(\Z;H)$ of integer matrices with entries in the interval $[-H,H]$, or sometimes bounded in some other norm, where these and other problems of similar 
  flavour have attracted quite significant interest; see~\cite{DRS,EskKat, EMS, GNY,  HOS, Kat1,Kat2, MOS, Shah, WX} 
  and  the references therein. 

In particular, Katznelson~\cite{Kat1,Kat2} has studied the distribution of singular matrices
matrices in $\cM_{n,n}(\Z;H)$ and more generally the distribution of matrices of 
prescribed rank from $\cM_{m,n}(\Z;H)$. 
Duke, Rudnick and Sarnak~\cite{DRS} have studied matrices with a fixed 
determinant $\delta \ne 0$ (also from $\cM_{n,n}(\Z;H)$), see also~\cite{BlLu,GNY}.  A uniform  with respect to $\delta$ upper
bound  of a right order of magnitude on  this quantity is given in~\cite{Shp1}, which has been recently improved in~\cite{Afif2} for the case $n=2$. Note that each 
of the above works~\cite{BlLu,DRS, GNY, Kat1, Kat2} actually provides an asymptotic formula for the corresponding counting problem, however, the matrices are ordered with respect to a slightly different norm. 

More recently, a nontrivial upper bound on the number of matrices from $\cM_{n,n}(\Z;H)$  with a given 
characteristic polynomial $f$, which is uniform with respect to $f$ has been given 
in~\cite{HOS}, see also~\cite{OS2}. In the case  where $f\in \Z[X]$ is a fixed irreducible polynomial, an asymptotic formula has been given by
Eskin, Mozes and Shah~\cite{EMS}. Further related results in this setting can be seen at~\cite{Shah,WX}. 

Furthermore, analogous questions have also been studied in the setting of finite fields~\cite{AhmShp, E-BLS, MOS}. 
For example, the question of counting matrices with a given characteristic polynomial,  where $\cM_n(\Q;H)$ is replaced by 
a similarly defined set $\cM_n(\F_q)$ of $n\times n$ matrices 
over the finite field $\F_q$ of $q$ elements, is completely solved by Reiner~\cite{reiner}.  The question of counting matrices in $\cM_n(\F_q)$ with a given rank is solved by Fisher~\cite{Fis}.

 The main difficulty of working over $\cM_n(\Q;H)$ or  $\cM_n(\Z^{-1};H)$ compared with, 
for example,  $\cM_n(\Z;H)$ 
 is the fact that the heights of the elements in $\cF(H)$ and $\cE(H)$  are usually not preserved linearly under addition. More specifically, the sum of two elements in those sets can result in a rational number whose height is of order $H^2$. Due to this restriction, some algebraic manipulations that could be easily performed in $\Z$ or $\F_q$ need to be done more carefully in this case.
 
 Another difficulty on working with these sets is that they do not seem to yield from the geometry of numbers methods (such as counting lattice points in convex bodies), as in the work of Katznelson~\cite{Kat1,Kat2} for the case of integer matrices.

Therefore,  here we use another approach, which stems from the idea of  Blomer and Li~\cite[Lemma~3]{BlLi} 
and which has then been improved in~\cite{MOS}.  This approach is based on counting solutions to some Diophantine 
equations and thus applies to counting matrices with  entries of very general nature
 (for example, parametrised by polynomials as in~\cite{BlLi,MOS}).  
More precisely, Blomer and Li~\cite[Lemma~3]{BlLi} have given an upper bound on the number of integral $m\times n$ matrices of rank $r$,  with entries  parameterised by monomials, see also~\cite{MOS} for  an improvement and generalisation of their 
approach and result. 

Here we employ similar ideas and thus we start with collecting some known, as well as deriving new, bounds 
on the number of solutions to linear equations in elements of $\cF(H)$ and $\cE(H)$, which we give in Section~\ref{eq:LinEq rat}  
and which can be of independent interest. 

Furthermore, for the case of small matrices ($n=2,3$), we may further improve our techniques by dealing directly with the entries and using other equations that arise from the matrix. It remains to see whether this approach can be utilised for larger $n$.

Generally, we do not have any reliable heuristic to make plausible conjectures about the behaviour of the above quantities 
$\cL_{m,n, r}\(\U;H\)$, $\cD_n(\U; H,\delta)$ and $\cP_n(\U; H,f)$, with $\U=\Q$ or $\Z^{-1}$. However in some cases we 
discuss what one may or may not expect, see Section~\ref{sec:Tight}. 

\subsection{Notation}
For a finite set $\cS$, we use $\# \cS$ to denote its cardinality. 
As  usual, the  equivalent 
notations
\[
U = O(V) \quad \Longleftrightarrow \quad U \ll V \quad \Longleftrightarrow \quad V\gg U
\] 
all mean that $|U|\le c V$ for some positive constant $c$, 
which through out this work may depend only on the dimension  $m$ and $n$.  We also write $U = V^{o(1)}$ if, for a given
$\varepsilon>0$, we have $V^{-\varepsilon}\le |U|\le V^\varepsilon$
for large enough $V$.

The letter $p$ always denotes a prime number.

\section{Main results}

\subsection{Bounds for $\Q$}

We begin with providing an upper bound on $\#\cL_{m,n, r}\(\Q;H\)$ with $n \ge m \ge r$. First, 
we record a lower bound
\begin{equation}
\label{eq:Lmnr Q low}
    \#\cL_{m,n,r}\(\Q;H\)\gg H^{2nr}
\end{equation} 
that comes from matrices in  $\cM_{m,n}\(\Q;H\)$ whose last $m-r$ rows are identical to their first row.
We also note that when $m=r$ we trivially have 
\[
     \#\cL_{r,n, r}\(\Q;H\) \ll H^{2nr}.
\] 
Therefore, we may assume that $m>r$.

 \begin{theorem}\label{thm:rank Q} For $n \ge m > r$ we have 
\[
\#\cL_{m,n, r}\(\Q;H\) \le \begin{cases}
          H^{2mr+nr+n-r^2-r+o(1) }, &\text{ if }2m \ge n+r,\\
          H^{2nr+2m-2r+o(1) }, &\text{ if } 2m <n+r.
      \end{cases}
\]    
\end{theorem}

In particular, for $n=m$,  Theorem~\ref{thm:rank Q} implies
\begin{equation}
\label{eq:Lnnr}
  \#\cL_{n,n, r}\(\Q;H\) \le H^{3nr+n-r^2-r+o(1)}.
\end{equation}

Next, we use  Theorem~\ref{thm:rank Q} to estimate $\#\cD_n\(\Q;H,\delta\)$.
 
 \begin{theorem}\label{thm:cdn} Uniformly over $\delta \in \Q$, we have
\[
\#\cD_n\(\Q;H,\delta\)\le \begin{cases}
   H^{4+o(1)},& \text{ if }n=2,\\
   H^{2n^2-n+o(1)},& \text{ if }n\ge 3.
   \end{cases}
\]  
    \end{theorem}

On the other hand, we note that we have a lower bound 
\begin{equation}
\label{eq:Triv DQ Sing}
\#\cD_n\(\Q;H,0\)\gg H^{2n^2-2n}.
\end{equation}
by considering $n\times n$ matrices $A\in \cM_n(\Q;H)$,  whose last row is all zero. Furthermore, 
we also have 
\begin{equation}
\label{eq:Triv DQ Non-Sing}
\#\cD_n(\Q;H,1)\ge H^{n^2 +o(1)}
\end{equation}
by considering all  upper diagonal matrices $A\in \cM_n(\Q;H)$ with 
\[
a_{i,i}=p_i/p_{i+1}, \qquad i =1, \ldots, n,
\]
where $p_1 <  \ldots < p_n\le H$ are primes and $p_{n+1} = p_1$. We note that these lower bounds match the bound of Theorem~\ref{thm:cdn} for the case $n=2$, up to the error term in $H^{o(1)}$. 

Again it is very interesting to reduce the gap between the upper bounds of Theorem~\ref{thm:cdn}
and the above lower bounds.

Next we give an upper bound on the cardinality of  $\cP_n(\Q;H,f)$. 

\begin{theorem}\label{thm:pnhf} Uniformly over $f\in \Q[X]$, we have 
\[
\#\cP_n\(\Q;H,f\)\ll\begin{cases}
	 H^{2+o(1)}, &\text{ if }n=2,\\
	  H^{n^2+o(1)}, & \text{ if }n\ge 3.
\end{cases}
\]
\end{theorem} 

We do not have a matching  lower bound on $\#\cP_n(\Q;H,f)$  for $n\ge 3$. However, if we fix the elements on the diagonal  and we let all entries above the diagonal vary in $\cF(H)$ and zeros under the diagonal (in which case, the fixed characteristic polynomial $f$ splits over $\Q$), then we get
\begin{equation}
\label{eq:Triv PQ}
\#\cP_n(\Q;H,f)\gg H^{n^2-n}.
\end{equation}
We note again that the bound~\eqref{eq:Triv PQ} matches the upper bound in Theorem~\ref{thm:pnhf} for $n=2$.

It is interesting to note that for $n \ge 3$ our upper bound in Theorem~\ref{thm:pnhf}  is based only on examining 
 $\Tr(A)$ and $\Tr\(A^2\)$ (that is, of only two coefficients of $f$); however, it is already 
surprisingly close to the lower bound when $n$ is large.

If the fixed characteristic polynomial $f$ is monic irreducible defined over $\Z$, then by counting matrices with integer entries and using a result of~\cite{EMS}, we get 
\[\#\cP_n(\Q;H,f)\gg H^{n(n-1)/2}.\]
It is certainly interesting 
to reduce the gap between these  lower bounds and 
the bound of Theorem~\ref{thm:pnhf}.  

\subsection{Bounds for $\Z^{-1}$}
 We start with presenting an upper bound on $\#\cL_{m,n, r}\(\Z^{-1};H\)$ with $n \ge m \ge r$. 
 We record first a lower bound
\begin{equation}
\label{eq:Lmnr 1/Z low}
    \#\cL_{m,n,r}\(\Z^{-1};H\)\gg H^{nr}
\end{equation} 
that comes from matrices in  $\cM_{m,n}(\Z^{-1},H)$ whose last $m-r$ rows are identical to their first row. 
 Therefore, as before, notice that we may assume that $m>r$. 
 
 \begin{theorem}\label{thm:rank 1/Z} For $n \ge m > r$ we have
\begin{itemize}
\item If $r=1$, \[
\#\cL_{m,n, 1}\(\Z^{-1};H\) \le  H^{n+o(1)}.
\]
\item If  $r=2$, 
 \[
\#\cL_{m,n, 2}\(\Z^{-1};H\) \ll \begin{cases}
        H^{7+o(1)},&\text{ if }(m,n)=(3,3),\\
        H^{2n+m-3+o(1)},&\text{ if }(m,n)\ne (3,3).
    \end{cases}\]
\item If $r \ge 3$, 
\[
     \#\cL_{m,n, r}\(\Z^{-1};H\)\ll \begin{cases}
 H^{(n-r)(r+1)/2+mr+o(1)},&\text{ if }2m\ge n+r,\\
 H^{nr+m-r},&\text{ if }2m< n+r.
     \end{cases} 
\]
\end{itemize}
 \end{theorem}

Our upper bounds for the case $r=1$ and the case $(m,r)=(3,2)$ and $n\ne 3$ match the lower bound~\eqref{eq:Lmnr 1/Z low}, up to the error factor $H^{o(1)}$.

Next we estimate 	$\#\cD_n\(\Z^{-1};H,\delta\)$, which is the main contribution of this section.

\begin{theorem}\label{thm:cdnE}
Uniformly over $\delta \in \Q$, we have:
	\begin{itemize}
		\item    If $\delta\ne 0$, then
		\[\#\cD_n\(\Z^{-1};H,\delta\)\le  \begin{cases}
			H^{o(1)},& \text{ if }n=2,\\
			H^{7+o(1)},& \text{ if }n=3,\\
			H^{n^2-n/2-1/(2n-2)+o(1)},& \text{ if }n\ge 4. 
		\end{cases}
		\]
		\item If $\delta=0$, then 
		\[
		\#\cD_n\(\Z^{-1};H,0\)\le \begin{cases}
			H^{2+o(1)},& \text{ if }n=2,\\
	 H^{7+o(1)},& \text{ if }n=3,\\
			H^{n^2-n/2+o(1)},& \text{ if }n\ge 4.
		\end{cases}
		\]
	\end{itemize}
\end{theorem}

On the other hand, we have a lower bound 
\begin{equation}
\label{eq:Triv DZ Sing}
\#\cD_n\(\Z^{-1};H,0\)\gg H^{n^2-n}
\end{equation}
by considering $n\times n$ matrices $A\in \cM_n(\Z^{-1};H)$,  whose first row is equal to the second row.  We note that this lower bound matches the bound of Theorem~\ref{thm:cdn} for the case $n=2$, up to the error factor in $H^{o(1)}$.

Next, we provide some upper bounds on $\#\cP_n\(\Z^{-1};H,f\)$.

\begin{theorem}\label{thm:pnhfE} We have
\begin{itemize}
\item If  $n=2$ and $f(X)=X^2$,   
\[ 
		\#\cP_2\(\Z^{-1};H,X^2\)=\frac{24}{\pi^2}H\log H + O(H).
\]
\item Uniformly over $f\in \Q[X]$, 
	\[
	\#\cP_n\(\Z^{-1};H,f\) \ll\begin{cases}
		H^{o(1)}, &\text{ if $n=2$ and  $f(X)\ne X^2$},\\
		H^{3+o(1)}, & \text{ if $n=3$},\\
		H^{n^2/2-1/(2n-2)+o(1)}, & \text{ if $n\ge 4$ and $f_{n-1}\ne 0$},\\
		H^{n^2/2+o(1)}, & \text{ if $n\ge 4$ and $f_{n-1}= 0$},
	\end{cases}
	\]
where $f_{n-1}$ is the coefficient of $X^{n-1}$ in $f$.
\end{itemize}
\end{theorem}

 \subsection{On the tightness of the bounds} 
 \label{sec:Tight}    
  We first discuss the tightness of the upper bound on $\#\cL_{m,n,r}(\Q;H)$ and $\#\cL_{m,n,r}(\Z^{-1};H)$. For both these quantities, for 
 $(n+r)/2\ge m>r$ by~\eqref{eq:Lmnr Q low}, \eqref{eq:Lmnr 1/Z low} 
 and Theorems~\ref{thm:rank Q} and~\ref{thm:rank 1/Z},   we have rather close upper and lower bounds 
 \[
 H^{2nr}\ll \#\cL_{m,n,r}(\Q;H) \le H^{2nr+2(m-r)+o(1)},
 \] 
 and  
  \[ 
 H^{nr}\ll \#\cL_{m,n,r}(\Z^{-1};H) \le H^{nr+m-r+o(1)}.
  \] 
 Furthermore, we have matching bounds on  $\#\cL_{m,n,r}(\Z^{-1};H) $  for several nontrivial triples $(m,n,r)$,  see  the comment after Theorem~\ref{thm:rank 1/Z}.

We remark that after clearing the denominators, the equation
 \[\det A=\delta, \qquad A\in \cM_{n,n}\(\Q;H\), \]  
leads to a homogeneous Diophantine equation of degree 
$n^2$ in $2n^2$ variables, which suggests the bound $\#\cD_n(\Q;H, \delta)\le H^{n^2 +o(1)}$
which matches~\eqref{eq:Triv DQ Non-Sing} for $\delta \ne 0$. On the other hand, the lower 
bound~\eqref{eq:Triv DQ Sing} shows that this conjecture is false for $\delta = 0$ in which case
the lower bound~\eqref{eq:Triv DQ Non-Sing} can be tight. A similar argument also suggests that 
  \[\#\cD_n(\Z^{-1} ;H, \delta)\le H^{n^2 -n +o(1)}, \]   
 which may actually be tight, see~\eqref{eq:Triv DZ Sing}.

Furthermore, based on the fact that the upper bound on $\#\cP_n\(\Q;H,f\)$ for $n \ge 3$ is based 
 on examining of only two coefficients of $f$, makes us believe that the lower bound~\eqref{eq:Triv PQ} is
 more precise and may in fact give the true order of $\#\cP_n\(\Q;H,f\)$. However we do not have any feasible 
 conjecture about the size of $\#\cP_n\(\Z^{-1};H,f\)$.

\section{Preliminary results}

 \subsection{Bounds on some arithmetic functions} 
 
 We use a well-known result concerning the number of $k$-fold divisors 
\[
 \tau_k(m) = \sum_{\substack{m_1, \ldots, m_k \in \N\\ m_1 \ldots m_k=m}} 1
 \]
  of a positive integer $m$, see, for example,~\cite[Equation~(1.81)]{IK}. In most of our applications in this paper, we use $k=2$. In this case, we sometimes drop the corresponding index; thus, $\tau_2=\tau$.

\begin{lemma}\label{lem:div} 
For all $k\ge 2$, we have 
\[
\tau_k (m)=m^{o(1)}
\] 
 as $m\to \infty$. \end{lemma}

Furthermore, for an integer $Q\ne 0$ and a real $U\ge 1$, let $F(Q,U)$  be the number 
of positive integers $u \le U$, whose prime divisors are amongst those of $Q$.

As in~\cite[Section~3.1]{PSSS}, in particular, see the derivation of~\cite[Equation~(3.10)]{PSSS},  using a result of  de Bruijn~\cite[Theorem~1]{dBr}, one immediately derives 
the following result, see also~\cite[Lemma~3.4]{HOS} for a self-contained proof.

\begin{lemma}
\label{lem: Sunits}
For any integer $Q\ne 0$ and   real $U\ge 1$, we have 
\[
F(Q,U) = \(Q U\)^{o(1)} .
\]
\end{lemma}

 \subsection{Squares in arithmetic progressions} 
 For integers $a$, $q$ and $K$, let $Q(a,q,K)$ denote the number of perfect integer
 squares in the progression $m^2 =   a + kq$, 
 with  a positive integer $k \le K$.
 
The classical conjecture of  Rudin~\cite{Rud} asserts that $Q(a,q,K) \ll K^{1/2}$ 
uniformly over $a$ and $q$. The conjecture is still open,  see also~\cite{BGP,BoZa}
for the state of art, however, it has been shown in~\cite[Section~I]{BGP} that it holds (up to 
a logarithmic factor)
uniformly over $q \le \exp\(\sqrt{K}\)$.

\begin{lemma}\label{lem:Rudin}
Uniformly over $1 \le q \le \exp\( \sqrt{K}\)$ and over any integer $a$,  we have 
\[
Q(a,q,K)  \ll K^{1/2} \log K. 
\]
\end{lemma} 
 
 \subsection{Counting solutions to some equations and congruences} 
 \label{eq:LinEq rat} 
To derive an upper bound on $\#\cD_n(\fA;H,\delta)$, where $\fA \in \{\Q, \Z^{-1}\}$, we first recall  a result of Shparlinski~\cite[Theorem~1.1]{Shp2} that concerns the number of rational solutions of a linear equation over $\Q$ of bounded height.

\begin{lemma}\label{lem:aixi-F}
Let $H_1,\ldots,H_n\ge 1$  and let $(Q_0,Q_1,\ldots, Q_n)\in \Z^{n+1}$ 
with $1\le |Q_i|\le \exp(H^{o(1)})$, $i=1,\ldots,n$, where  $H = \max_{i=1}^n H_i$.  
Then, the equation 
\[
    \sum_{i=1}^n Q_ix_i = Q_0,
\]
has at most $H_1\cdots  H_n(\log H)^{2^n-1+o(1)}$ solutions with $x_i \in \cF(H_i)$, $i=1,\ldots,n$.  
\end{lemma}

In particular, if $H_1=\cdots=H_n=H$, then the equation  of Lemma~\ref{lem:aixi-F} 
has at most $H^{n+o(1)}$ solutions $(x_1,\ldots x_n) \in \cF(H)^n$.

We note that de la Bret{\`e}che~\cite{dlBr} has obtained more precise results.
However, the results are not uniform with respect to the coefficients $Q_0,Q_1,\ldots, Q_n$, 
which is crucial for our applications. 

We obtain a similar result for linear equations in Egyptian fractions, which can be of independent interest.
\begin{lemma}\label{lem:aixi-Z} 
Let $(Q_0,Q_1,\ldots, Q_n)\in \Z^{n+1}$ with $1\le |Q_i|\le H^{O(1)}$ for  $i=1,\ldots,n$. 
Then, the equation 
 \[
\sum_{i=1}^n Q_i/x_i = Q_0,
\] 
has at most $H^{n/2+o(1)}$ solutions $(1/x_1,\ldots, 1/x_n) \in \cE(H)^n$. 
\end{lemma} 

\begin{proof} We use some ideas from~\cite{Kar}. Namely, writing the above equation as
\begin{equation}
\label{eq:Lin eq X/xi}
\sum_{i=1}^n Q_i X/x_i = Q_0 X,
\end{equation}
where $X = x_1 \ldots x_n$, we see that for every  $i,j=1,\ldots , n$ we have
\begin{equation}
\label{eq:xj X/xi ori}
 x_j \mid Q_i X/x_i  .
\end{equation}
Indeed, the above divisibility is obvious for $i \ne j$, but then the equation~\eqref{eq:Lin eq X/xi}
also implies it for $i=j$. 

Let 
\[
L = \lcm[x_1 \ldots x_n] \mand 
Q = Q_1\ldots Q_n.
\]
We see from~\eqref{eq:xj X/xi ori} that for every  $i=1, \ldots, n$ we have 
\[
Lx_i \mid Q X,
\]
which in turn implies that 
\[
L^2 \mid Q X  .
\]
Now, we choose some $d \le H^n$, which is factored into primes which are prime divisors of $Q$, and count the number
of solutions $(1/x_1,\ldots 1/x_n) \in \cE(H)^n$ to~\eqref{eq:Lin eq X/xi} with 
$L = d K$ where $\gcd(K, Q)=1$. Then $K^2 \mid X$ and thus $K \le H^{n/2}$. Since
$x_i \mid L = dK$,   $i=1, \ldots, n$  by Lemma~\ref{lem:div}, we see that for each $K \le H^{n/2}$
and each $d$ we have $H^{o(1)}$ choices for $(x_1, \ldots, x_n)$. 

Since by  Lemma~\ref{lem: Sunits} we have $(HQ)^{o(1)} = H^{o(1)}$ choices for $d$, the result now 
follows. 
\end{proof}

\begin{rem}
\label{rem:tightness}  
If $n$ is even, the upper bound in 
Lemmas~\ref{lem:aixi-F} and~\ref{lem:aixi-Z} is tight (up to the error factor $H^{o(1)}$) as the choice of 	
$Q_0=0$ and $Q_{2r-1}=-Q_{2r}$ for $1\le r \le n/2$, shows.  However,  for odd $n$ a similar construction only produces equations with at least $H^{(n-1)/2}$ solutions in Lemma~\ref{lem:aixi-Z}. 
\end{rem}

 We now improve Lemma~\ref{lem:aixi-Z} for $n = 2$ and $n = 3$.

\begin{lemma}\label{lem:aixi-Z n=2,3} 
Let $(Q_0,Q_1,\ldots, Q_n)\in \Z^{n+1}$ with $1\le |Q_i|\le H^{O(1)}$,  $i=1,\ldots,n$. 
Then, the equation 
 \[
\sum_{i=1}^n Q_i/x_i = Q_0,
\] 
has at most 
\begin{itemize}
	\item one solution $1/x_1 \in \cE(H)$ if $n=1$,
\item   $H^{o(1)}$ or  $O(H)$ solutions $(1/x_1, 1/x_2) \in \cE(H)^2$,
depending on whether $Q_0\ne 0$  or $Q_0 = 0$  if $n=2$, 
\item $H^{1+o(1)}$ solutions $(1/x_1, 1/x_2,  1/x_3) \in \cE(H)^3$, if $n=3$. 
\end{itemize}
\end{lemma} 

\begin{proof} The case $n=1$ is trivial. Next, we transform the equation
\[
\sum_{i=1}^n Q_i/x_i = Q_0,
\] 
to 
\[
\sum_{i=1}^n \frac{1}{(Q/Q_i)x_i} = \frac{Q_0}{Q},
\] 
where $Q=Q_1\ldots Q_n$.

Then for $n=2$ and $Q_0\ne 0$ the bound is instant from  a result of Browning and Elsholtz~\cite[Theorem~1]{BE}. If $Q_0=0$, fixing $x_2$ generates at most one possible $x_1$, which completes the proof.

For $n=3$, when $Q_1/x_1+Q_2/x_2=0$, we have a unique choice for $x_3$ with 
\[
 \frac{1}{(Q/Q_3)x_3} = \frac{Q_0}{Q}
\]
after which we obviously have $O(H)$ choices for  $(1/x_1, 1/x_2,  1/x_3) \in \cE(H)^3$. 
For the remaining $O(H)$ choices  of $x_3$ we apply the above result with $n=2$. 
\end{proof} 

Next, we improve Lemma~\ref{lem:aixi-Z} in the inhomogeneous case (that is, when $Q_0 \ne 0$)
and $n\ge 4$.
\begin{lemma}\label{lem:aixi-Z-inhom} 
Let $n\ge 4$ and  let  $(Q_0,Q_1,\ldots, Q_n)\in \Z^{n+1}$ with $1\le |Q_i|\le H^{O(1)}$,  $i=0,\ldots,n$. 
Then, the equation 
\begin{equation}
\label{eq:Qixi}
\sum_{i=1}^n Q_i/x_i = Q_0,
\end{equation}
has at most $H^{n/2 -1/(2n-2)+o(1)}$ solutions $(1/x_1,\ldots, 1/x_n) \in \cE(H)^n$. 
\end{lemma}

\begin{proof}
 We note that our approach borrows some ideas from~\cite{KK1,KK2}, with the main difference being we are only interested in the order of the expressions (and so we replace all logarithmic terms with $H^{o(1)}$).
 
We first note that the the number of solutions $(1/x_1,\ldots,1/x_n)\in\cE(H)^n$ of  the equation~\eqref{eq:Qixi} satisfying
\[
\frac{Q_i}{x_i}+\frac{Q_j}{x_j}=Q_0
\]
for some $1\le i, j\le n$ is at most 
\begin{equation}
\label{eq:ij}
H^{n/2-1/2+o(1)}. 
\end{equation}
Indeed, if $x_i,x_j$ satisfies the equation above for some $1\le i, j\le n$, then, if $i=j$, $x_i$ is uniquely defined. If $i\ne j$, since $Q_0\ne 0$, by Lemma~\ref{lem:aixi-Z n=2,3} there are $H^{o(1)}$ solutions $(1/x_i,1/x_j)\in\cE(H)^2$. For each such choice of $x_i,x_j$, we apply Lemma~\ref{lem:aixi-Z} to~\eqref{eq:Qixi} seen now as an equation in $n-1$ or $n-2$ variables depending on $i=j$ or $i\ne j$, respectively. Thus, we obtain at most $H^{(n-1)/2+o(1)}$ or $H^{(n-2)/2+o(1)}$ solutions, respectively, which concludes our claim.

We now  count solutions of the equation~\eqref{eq:Qixi} with 
 \begin{equation}
 \label{eq:Qixi+Qjxj}
\frac{Q_i}{x_i} + \frac{Q_j}{x_j} \ne Q_0,\qquad  1 \le i,j \le n. 
\end{equation}

Furthermore, by changing the signs of $Q_1, \ldots, Q_n$,  without loss of generality, we may assume $x_i$ are all positive, for $1\le i \le n$. 

We proceed now as in the proof of Lemma~\ref{lem:aixi-Z} and exactly as therein we get to
\begin{equation}\label{eq:LQX}
L^2 \mid Q X,
\end{equation}
where
\[
X=x_1\cdots x_n,\quad L = \lcm[x_1 \ldots x_n],\quad  
Q = Q_1\ldots Q_n.
\]

From now on, we are counting the number of tuples $(x_1,\ldots,x_n)\in\Z^n$ with $0<x_i\le H$ for $i=1,\ldots,n$ that satisfy~\eqref{eq:LQX} for a fixed $Q$.

We let $L=qK$,  with integers $q$ and $K$ defined by the conditions
\[\rad q \mid \rad Q \mand \gcd(K,Q)=1, 
\] 
where $\rad k$ denotes the product of all prime divisors of an 
integer $k\ne 0$.

 From~\eqref{eq:LQX}, since $L$ and $X$ have same prime divisors, we see that each prime divisor $p$ of $L$ that is relatively prime to $Q$ appears in the prime factorization of $X$ with exponent at least 2. Therefore, we may write
 \[
    X=u^3v^2w,
\]
with $\rad w \mid \rad Q$, $\gcd(uv,Q)=1$, and $u$ is squarefree.

We now fix $q$ and $w$. Since both of them have the same prime divisor as $Q$, we may recall Lemma~\ref{lem: Sunits} to imply that there are $Q^{o(1)}\le H^{o(1)}$ possible choices for both $q$ and $w$.

Let $U>0$ be a parameter which is to be chosen at the end of the proof.
We divide now our counting based on the case $u> U$ and $u\le U$. We denote by $J_1$ and $J_2$, respectively, the number of tuples $(x_1,\ldots,x_n)\in\Z^n$ with $0<x_i\le H$, $i=1,\ldots,n$, satisfying~\eqref{eq:LQX} that correspond to $u> U$ and $u\le U$, respectively.

We first consider the case $u>U$. First, we note that $u^3v^2 \le X \le H^{n}$, which implies $v \le  (H^n/u^3)^{1/2}$ and $u\le H^{n/3}$.

We note that since $x_1\ldots x_n=u^3v^2w$, by Lemma~\ref{lem:div}  we see that that each choice of $u$ and $v$ corresponds to at most $\tau_n(u^3v^2w) = H^{o(1)}$ tuples $(x_1,\ldots,x_n)$.
Recalling that $U< u\le H^{n/3}$ and $v \le  (H^n/u^3)^{1/2}$, we have that the number of solutions to the equation~\eqref{eq:Qixi}, in this case, $J_1$, satisfies the upper
 bound \begin{equation}\label{eq:J1}
\begin{split}
    J_1 \le&    H^{o(1)} \sum_{U<u \le H^{n/3}} \sum_{v \le  (H^n/u^3)^{1/2}}  1 \\
    \le & H^{n/2+o(1)}\sum_{u>U} \frac{1}{u^{3/2}}  =  H^{n/2+o(1)} U^{-1/2}.
    \end{split}
\end{equation} 

Next, we consider the case $u\le U$. For $i=1,\ldots,n$, we write 
\[
    x_i=w_iy_i,
\]
for some $w_i$ with $\rad w_i \mid \rad Q$ and $\gcd(y_i,Q)=1$. We also define 
\[
Y=y_1\ldots y_n,
\]
and note that $Y=u^3v^2=Xw^{-1}$.  

Since $\rad w_i \mid \rad Q$, we can apply Lemma~\ref{lem: Sunits} to conclude that we have at most $Q^{o(1)}\le H^{o(1)}$ choices for $w_i$. We now fix such a choice of $w_i$, for $i=1,\ldots,n$.

We now rewrite the equation~\eqref{eq:Qixi} as
\begin{equation}
\label{eq:Qi/wi}
\sum_{i=1}^n \frac{Q_i/w_i}{y_i} = Q_0.
\end{equation}

Without loss of generality, we may assume $y_1 \ge y_i$ for all $2\le i \le n$. This implies $y_1\ge Y^{1/n}$.

We now show that there exist an index $2\le j \le n$ with
 \begin{equation}\label{eq:large gcd}
 \gcd(y_1,y_j)\ge Y^{1/n(n-1)}.  
\end{equation} 
Indeed, from~\eqref{eq:LQX}, we have
 \begin{equation}\label{eq:y1}
    y_1 \mid y_2\ldots y_n.
\end{equation} 
Next, we prove that 
\begin{equation}\label{eq:gcd}
    y_1\mid \prod_{i=2}^n \gcd(y_1,y_i).
\end{equation} 
To prove this relation, we first consider a prime $p$ and denote by $v_p(k)$ the $p$-adic valuation 
of an integer $k \ne 0$.  If for some $2\le i \le n$ we have $v_p(y_i)\ge v_p(y_1)$, we  have $v_p(y_1)=v_p(\gcd(y_1,y_i))$.
Otherwise, we have 
\[
v_p(y_1)>v_p(y_i)
\]
for all $2\le i \le n$. In this case, from~\eqref{eq:y1}, we have 
\[
    v_p(y_1) \le \sum_{i=2}^n v_p(y_i) = \sum_{i=2}^n v_p(\gcd(y_1,y_i)). 
\]
In either case, we have 
\[v_p(y_1)\le   \sum_{i=2}^n v_p(\gcd(y_1,y_i))
\]
 for all primes $p$, which implies~\eqref{eq:gcd}.

To finish up, we note that~\eqref{eq:gcd} implies
\[
\(\max_{2\le i \le n}\gcd(y_1,y_i)\)^{n-1} \ge \prod_{i=2}^n \gcd(y_1,y_i) \ge y_1 \ge Y^{1/n}.
\]
Hence, for $j$  with $\gcd(y_1,y_j)=\displaystyle \max_{2\le i \le n}\gcd(y_1,y_i)$ we obtain~\eqref{eq:large gcd}. 

Let $d=\gcd(y_1,y_j)$, where $j$ satisfies the inequality~\eqref{eq:large gcd}. Then, we have $d\ge Y^{1/n(n-1)}$. Since $d^2\mid Y=u^3v^2$ and $u$ is squarefree, we also have 
\[
d\mid uv.
\] Let $uv=d\ell$, and define $z_1=y_1/d$, $z_j=y_j/d$. We have
 \begin{equation}\label{eq:z1zj}
    z_1z_j\prod_{i=2, i\neq j}^n y_i =  Yd^{-2}=u\ell^2,
\end{equation} with 
\[
    \ell \le \frac{Y^{1/2}}{du^{1/2}} \le Y^{\frac{1}{2}-\frac{1}{n(n-1)}}u^{-1/2}.
\]

We now consider the quantity $u\ell^2$ in the equation~\eqref{eq:z1zj}. For a fixed $u$ and $\ell$, we note that since $z_1\mid u\ell^2$, there are $(u\ell^2)^{o(1)}\le H^{o(1)}$ possible values of $z_1$. Applying the same argument for other variables, we see that there are at most $H^{o(1)}$ possible tuples 
\[(z_1,z_j,y_2,\ldots,y_{j-1},y_{j+1},\ldots,y_n),
\] for a fixed $u$ and $\ell$. We note that the number of possible different values of $u\ell^2$ is at most 
\[
    \sum_{u\le U} Y^{\frac{1}{2}-\frac{1}{n(n-1)}}u^{-1/2} \ll Y^{\frac{1}{2}-\frac{1}{n(n-1)}}U^{1/2}.
\] 
 Therefore, there are at most 
\[
    H^{o(1)}Y^{\frac{1}{2}-\frac{1}{n(n-1)}}U^{1/2} \le H^{n/2- 1/(n-1)+o(1)}U^{1/2} 
\] 
possible tuples $(z_1,z_j,y_2,\ldots,y_{j-1},y_{j+1},\ldots,y_n)$.

We now rewrite the equation~\eqref{eq:Qi/wi} as 
\[
\sum_{\substack{i=2\\i\neq j}}^n \frac{Q_i/w_i}{y_i} = Q_0-\frac{1}{d} \( \frac{Q_1/w_1}{z_1}+\frac{Q_j/w_j}{z_j}\).
\]
Recalling the condition in the equation~\eqref{eq:Qixi+Qjxj}, we see that for a given 
\[
(z_1,z_j,y_2,\ldots,y_{j-1},y_{j+1},\ldots,y_n),
\] 
the above equation determines $d$ uniquely. This, in turn, gives a solution $(x_1,\ldots ,x_n)$ of the original equation~\eqref{eq:Qixi}, in this case. 

To conclude, we note that the number of solutions $J_2$ to the equation~\eqref{eq:Qixi} with $u \le U$, satisfies the upper bound 
\begin{equation}\label{eq:J2}
    J_2 \le H^{n/2- 1/(n-1)+o(1)}U^{1/2} .
\end{equation}
Putting together the inequalities~\eqref{eq:J1} and~\eqref{eq:J2}, we have
\[
    J_1+J_2 \le H^{n/2+o(1)} U^{-1/2} + H^{n/2- 1/(n-1)+o(1)}U^{1/2}.
\]
Letting $U=H^{1/(n-1)}$, we have 
\[
    J_1+J_2 \le H^{n/2- 1/(2n-2) +o(1)}, 
\]
and taking also into account~\eqref{eq:ij}, we complete the proof.
\end{proof}

\begin{rem} 
We note that several variations of  the equations with Farey and Egyptian fractions, have also been considered several other works, see~\cite{BlKu, dlBr, Dest} for further results, references and applications to other problems. However, the lack of uniformity with respect to the coefficients of the corresponding equations precludes their applications to the questions we study here.
\end{rem}

We also need a result of de la Bret{\`e}che, Kurlberg and Shparlinski~\cite[Corollary~2.4]{dlBKS} that gives an asymptotic formula for the number $N(H)$ of  pairs $(a,b)$ 
of positive   integers with $a,b\le H$, such that their product $ab$ is a perfect integer square, which is used in the proof of Theorem~\ref{thm:pnhfE}.

\begin{lem}\label{lem:square} We have
 \[
		N(H)=\frac{6}{\pi^2}H\log H + O(H).
\]	
\end{lem}

\section{Proof of Theorem~\ref{thm:rank Q}}

As in~\cite{BlLi}, and also in~\cite{MOS},  we observe that  it is enough to estimate the number $\#\cL_{m,n, r}^*(\Q;H)$ 
of matrices $A \in \cM_{m,n}\(\Q;H\)$ which are of rank $r$ and such that the top left $r \times r$ 
minor $A_r = \(a_{i,j}\)_{1\le i,j\le r} $ is non-singular.  We now fix the values of 
such $x_{i,j}$, $  i,j=1, \ldots, r$, in 
\begin{equation}
\label{eq:contrib Xr}
\cI \ll H^{2r^2} 
\end{equation}
ways. 

We now observe that  for every integer  $h$, $r < h \le m$, once the minor $A_r$ is
fixed, the $h$-th   row of every matrix  $A$, which is counted by  $\#\cL_{m,n, r}^*(\Q;H)$ is a unique
linear combination of the first $r$ rows with  coefficients $\(\rho_{1}(h) , \ldots , \rho_{r}(h)\) \in \Q^r$.

We say that the $(m-r)\times r$ matrix  
\[
B_r =\(a_{h,j}\)_{\substack{r+1\le h\le m \\ 1 \le j\le r}}
\] 
(which is directly under $A_r$ in $A$) is of \textit{type $t\ge 0$} if $t$ is the largest number of  non-zeros among 
the coefficients $\(\rho_{1}(h) , \ldots , \rho_{r}(h)\) \in \Q^r$ taken over all 
$h = r+1, \ldots, m$. We note that $t\le r < \min(m,n)$. In particular, the only matrix $B_r$ that is of type $0$ is the zero matrix. 

Clearly if the  $h$-th   row is of type $t$, that is, $t$ of the coefficients in $\(\rho_{1}(h) , \ldots , \rho_{r}(h)\) \in \Q^r$
 are non-zero, say  $\rho_{1}(h) , \ldots , \rho_{t}(h) \ne 0$, then we can choose a $t \times t$ non-singular  
sub-matrix  of the matrix  $\(a_{i,j}\)_{\substack{1\le i\le t \\ 1 \le j\le r}}$. 
Again, without loss of generality we can assume 
that this is 
\[ 
A_t =  \(a_{i,j}\)_{1\le i, j\le t}.
\] 
This means that each of $O(H^{2t})$ possible choices 
of $a_{h,1},\ldots,a_{h,t}$  defines the coefficients  
\[
\( \rho_{1}(h) , \ldots , \rho_{r}(h)\) = \( \rho_{1}(h) , \ldots , \rho_{t}(h), 0, \ldots, 0\)
\]
and hence the rest of the values $a_{h,j}$, $j = t+1, \ldots, r$. 
Note that this bound is monotonically increasing with $t$ and thus applies to every row
of  matrices $B_r$ of type $t$. 

Therefore, for each fixed $A_r$, there are 
\begin{equation}
\label{eq:contrib Yr-t}
\cJ_t \ll H^{2t(m-r)}
\end{equation}
matrices $B_r$ of type $t$. We note that this bound is also true for $t = 0$.

Let now a matrix $A_r$ and a matrix $B_r$ of type $t$ be  both fixed. In particular, we note that the first $r$ columns of $A$ are fixed. Applying the same argument in our analysis on $B_r$, we see that there is an $h$ with $r+1\le h \le m$, such that the $h$-th row of $A$ can be written as a linear combination of the top $r$ rows. 

 As before, without loss of generality we can assume that the  vector
$ \( \rho_{1}(h) , \ldots , \rho_{r}(h)\) $ contains exactly $t$ non-zero components, 
which  for each $j = r+1, \ldots, n$ leads to an equation 
\begin{equation}
\label{eq:lin rel}
\rho_{1}(h) a_{1,j}  + \ldots  + \rho_{r}(h)   a_{r,j}  = a_{h,j}
\end{equation}
with exactly $t+1$ non-zero coefficients, and $a_{1,j}\ldots ,a_{r,j},a_{h,j}\in \cF(H)$.

We now see that the entries $a_{i,j}$, for $i \in \{r+1, \ldots, m \} \setminus \{h\}$ and $j=r+1,\ldots,n$, are uniquely defined. Indeed, for each $i \in \{r+1, \ldots, m \} \setminus \{h\}$, we have the analogue of the equation~\eqref{eq:lin rel} with different indices and coefficients to determine $a_{i,j}$. In this case, since the corresponding left-hand side is already fixed, there is at most one possible entry $a_{i,j}\in \cF(H)$.

We now derive an upper bound on the number of possible entries of the last $n-r$ columns of $A$, for a fixed $A_r$ and $B_r$ of type $t$. By the above discussion, this quantity is bounded above by $\cK_t^{n-r}$, where $\cK_t$ is the largest number of solutions to the equation~\eqref{eq:lin rel} in $\cF(H)^{t+1}$ for $j \in \{r+1, \ldots, n \}$.

To bound $\cK_t$ for $t\ge 1$, we apply Lemma~\ref{lem:aixi-F} to~\eqref{eq:lin rel}, and the trivial bound for the entries $a_{i,j}$ that correspond to the indices for which $\rho_i(h)=0$.

 We thus have 
\begin{itemize}
\item  $ \cK_t \le \# \cF(H)^r \ll H^{2r}$ if $t=0$;  
\item   $ \cK_t \le  \#\cF(H)^{r-t} H^{t+1+o(1)} \le  H^{2r - t+1+o(1)}$,  which follows from 
Lemma~\ref{lem:aixi-F}  if $  t \ge 1$. 
\end{itemize}

Now we are ready to derive an upper bound on $\#\cL_{m,n,r}^*(\Q;H)$ from these arguments. From our definitions of $\cI$, $\cJ_t$, and $\cK_t$, we have
\begin{equation}
\label{eq: L vs ABC}
\#\cL_{m,n, r}^*(\Q;H)   \le \cI   \sum_{t =0}^r \cJ_t \cK_t^{n-r}.
\end{equation}

We now analyse the contributions of each term of the right-hand side of~\eqref{eq: L vs ABC}.
In particular, we recall the bounds~\eqref{eq:contrib Xr} and~\eqref{eq:contrib Yr-t}
as well as the above bounds on $\cK_t$. 

For $t =0 $ the total contribution $\fL_{0}$ satisfies
\begin{equation}\label{eq:t = 0}
\#\fL_{0} \ll H^{2r^2} \(H^{2r}\)^{n-r} = H^{2nr}. 
\end{equation} 

For $t \ge 1$ the total contribution $\#\fL_{\ge 1}$ satisfies 
\begin{equation}\label{eq:t ge 1}
\begin{split}
\#\fL_{\ge 1}&\le H^{2r^2}\sum_{t =1}^r   H^{2t(m-r)} \( H^{2r - t+1+o(1)}\)^{n-r}  \\
& = H^{2r^2+(2r+1) (n-r)+o(1) }\sum_{t =1}^r   H^{t(2m-n-r)}\\
& = H^{ 2nr+n-r+o(1) } \( H^{2m-n-r}+ H^{r(2m-n-r)}\)\\
& = H^{2r(n-1)+2m+o(1) }  +  H^{r(2m+n-r-1)+n+o(1) }  .
\end{split}  
\end{equation}

Substituting the bounds~\eqref{eq:t = 0} 
and~\eqref{eq:t ge 1} in the inequality 
\[
 \#\cL_{m,n, r}\(\Q;H\)  \ll \#\cL_{m,n, r}^*(\Q;H) \le \#\fL_{0}   
+ \#\fL_{\ge 1} , 
\] 
we have
\begin{align*}
 \#\cL_{m,n, r}\(\Q;H\) &\le H^{2nr } +   H^{2r(n-1)+2m+o(1) }  +  H^{r(2m+n-r-1)+n+o(1) }    \\
&  \le  H^{2nr+2m-2r+o(1) }  +  H^{2mr+nr+n-r^2-r+o(1) }.
	\end{align*}

Next,  we compare these two terms. We have that 
\[
     H^{2mr+nr+n-r^2-r}/ H^{2nr+2m-2r}= H^{(r-1)(2m-n-r)}.
\]
Therefore, we have that 
\[
      \#\cL_{m,n, r}\(\Q;H\) \le \begin{cases}
          H^{2mr+nr+n-r^2-r+o(1) }, &\text{ if }2m \ge n+r,\\
          H^{2nr+2m-2r+o(1) }, &\text{ if } 2m <n+r.
      \end{cases}
\]
This  concludes the proof of Theorem~\ref{thm:rank Q}.
 
\section{Proof of Theorem~\ref{thm:cdn}}

\subsection{The case $n=2$}
Let $\delta=r/s$ be a rational number, with integers $r$ and $s\ge 1$ satisfying $\gcd(r,s)=1$. We first note that $\delta$ (and consequently $r$) is allowed to be zero in this case.

Let 
\[A= \begin{pmatrix}
    a_1/b_1 & a_2/b_2 \\ a_3/b_3 & a_4/b_4
\end{pmatrix}\in \cM_2(\Q;H)
\] be a $2\times 2$ matrix with $\det A= \delta$ and $\gcd(a_i,b_i)=1$ for $i=1,2,3,4$. By using the determinant formula and clearing denominators, we have the following equation: 
\begin{equation}\label{eq:a1b1}
rb_1b_2b_3b_4+sa_2a_3b_1b_4 = s a_1a_4b_2b_3. 
\end{equation} 
First, suppose that none of $a_1,a_2,a_3,a_4$ is zero. We  can fix $(a_1, a_4, b_2, b_3)$ in $O(H^4)$ possible ways. 
Now, we see from~\eqref{eq:a1b1} that 
\[
    b_1b_4 \mid sa_1a_4b_2b_3.
\]
 By Lemma~\ref{lem:div}, this implies that there are $(sa_1a_4b_2b_3)^{o(1)} \le H^{o(1)}$ possible choices of $b_1$ and $b_4$ each.
 With  $(a_1, a_4, b_1, b_2, b_3, b_4)$ fixed, we see from~\eqref{eq:a1b1}  that 
 \[
 a_2a_3\mid  s a_1a_4b_2b_3 - rb_1b_2b_3b_4.
\] 
Applying Lemma~\ref{lem:div} again, we see that there are $H^{o(1)}$ possible choices for $a_2$ and $a_3$ each.

By multiplying all bounds, we see that there are at most $H^{4+o(1)}$ possible choices for $a_i$ and $b_i$, for $i=1,2,3,4$, in this case.

Now, suppose that at least one of $a_1$, $a_2$, $a_3$, and $a_4$ is zero. Without loss of generality, let $a_3=0$. Since $\gcd(a_3,b_3)=1$, 
we may assume $b_3=\pm 1$. 
Since $b_2\ne 0$, we can rewrite the equation~\eqref{eq:a1b1} as 
\[
    sa_1a_4=rb_1b_4.
\]
If both $a_1$ and $a_4$ are nonzero (which is possible in $O(H^2)$ ways), we see that 
\begin{equation}\label{eq:a1a4}
    b_1b_4\mid sa_1a_4.
\end{equation}
 Applying Lemma~\ref{lem:div}, we conclude from~\eqref{eq:a1a4} that there are $H^{o(1)}$ possible choices for $b_1$ and $b_4$ each. For each such choice $(a_1,a_4,b_1,b_4)$ we  see that there are $O(H^2)$ ways to choose the remaining variable $a_2$ and $b_2$. Therefore, in this case, there are  also at most $H^{4+o(1)}$ ways of choosing  $a_i$ and $b_i$, for $i=1,2,3,4$.

If  $a_1=0$ or $a_4=0$,  then as before we conclude that $b_1=1$ or  $b_4=1$ and thus we have $O(H^4)$ 
possibilities for the remaining $4$ variables. 
 
Combining all cases, we see that there are at most $H^{4+o(1)}$ possible choices of a $2\times 2$ matrix $A$ such that $\det A=\delta$. This completes the proof.

\subsection{The case $n\ge 3$} \label{laplace}

We first consider  the problem of bounding $\#\cD_n(\Q;H,0)$, the number of $n\times n$ matrices with entries in $\cF(H)$ and of determinant 0. We note that all matrices in the corresponding set have rank at most $n-1$. Therefore, recalling~\eqref{eq:Lnnr} we have 
\begin{equation}
\label{eq:det 0}
\begin{split}
    \#\cD_n(\Q;H,0) & \ll \sum_{r=0}^{n-1} \#\cL_{n,n,r} (\Q;H)\\
    &  \le \sum_{r=1}^{n-1} H^{r(3n-r-1)+n+o(1) }  \le H^{2n^2-n+o(1)},
 \end{split} 
\end{equation}
 where  we have used that $r(3n-r-1)$ is an increasing function of $r$ for $r\in [1,n-1]$. This completes the proof in the case of $\#\cD_n(\Q;H,0)$. 

We now consider the problem of bounding $\#\cD_n(\Q;H,\delta)$, with $\delta\ne 0$. Consider an $n\times n$ matrix $A=\(a_{i,j}\)_{i,j=1}^n$ with entries in $\cF(H)$ and of determinant $\delta$. 
Fix the last $n-1$ rows of $A$. By using the Laplace expansion on the first row of $A$, we have 
\begin{equation}\label{eqn:Lap}
    \delta = \det A = \sum_{j=1}^n (-1)^{1+j}a_{1,j}\det M_{1,j},
\end{equation} 
where $M_{1,j}$ is  the minor of $A$ resulted by removing the first row and $j$-th column of $A$.

As the last $n-1$ rows of $A$ are fixed, $(-1)^{1+j}\det M_{1,j}$ is a fixed rational number for $j=1,\ldots,n$. 

Let $L_j$ be the least common multiple of the denominators 
of all entries of $M_{1,j}$, $j =1, \ldots, n$, and let 
\[
L = \lcm[L_1, \ldots, L_n] \mand Q_0 = L\delta.
\]
Clearly $L\le H^{n(n-1)}$. 

We now set 
\[
Q_j =  (-1)^{1+j}   L \det M_{1,j}, \qquad j=1,\ldots,n,
\]
and rewrite~\eqref{eqn:Lap} as
 \begin{equation}\label{eqn:Lap2}
 \sum_{j=1}^n Q_j a_{1,j} = Q_0.
 \end{equation}
 
 Clearly 
\[
 Q_j\in \Z \quad \text{and} \quad  |Q_j| \le H^{O(1)}, \qquad j=0,\ldots,n.
\]  
We first define 
\begin{align*}
	\widetilde \cD_n&\(\Q;H,k,\delta\)\\
& =\{A\in \cD_n\(\Q;H,\delta\):~ \text{exactly } k \text{ of } Q_1,Q_2,\ldots,Q_n\text{ are nonzero}\}. 
\end{align*}
Since $\delta \ne 0$, we can assume  that $k\ge 1$ as 
otherwise~\eqref{eqn:Lap} has no solution. 

Trivially, we have that 
\[	
\cD_n\(\Q;H,\delta\) =\bigcup_{k=1}^n	\widetilde \cD_n(\Q;H,k,\delta).
\]

To estimate $\# \widetilde \cD_n(\Q;H,k,\delta)$, 
without loss of generality we can only consider $A \in \widetilde \cD_n(\Q;H,k,\delta)$ 
with  $Q_0, Q_1, \ldots, Q_k\ne 0$.

We note that the equation~\eqref{eqn:Lap2} satisfies the conditions of Lemma~\ref{lem:aixi-F}.

We first consider the case $k = n$ and bound $\# \widetilde \cD_n(\Q;H,n,\delta)$. In this case,  we can directly apply Lemma~\ref{lem:aixi-F} to conclude  that the equation~\eqref{eqn:Lap} has at most $H^{n+o(1)}$ solutions. 
This implies that for each  possible  $(n-1)\times n$ bottom  sub-matrix of $A$, there are at most $H^{n+o(1)}$ possible choices for the first row of $A$. Since we can choose the bottom $n-1$ rows in  
$O\(H^{2(n-1)n}\)$ ways, we see that there are at most
\begin{equation}
\label{eq:k=n}
 \widetilde \cD_n(\Q;H,n,\delta)   \le
H^{n+o(1)} \cdot H^{2(n-1)n} \le H^{2n^2-n+o(1)}\end{equation}
possible matrices $A$ with $\det A = \delta$ in this case.

Now we give an upper bound for $\# \widetilde \cD_n(\Q;H,k,\delta)$ for all $1\le k <n$. Let $A\in  \widetilde \cD_n(\Q;H,k,\delta)$ for some $k< n$.  In particular, we have $Q_n=0$ and $Q_1\ne 0$. The first equality implies $\det M_{1,n}=0$. Hence, there are  $\# \cD_{n-1}(\Q;H,0)$ possibilities for the $(n-1)\times (n-1)$ submatrix $M_{1,n}$ of $A$ that has determinant zero. Moreover, there are at most $H^{2n}$ ways of choosing the last column $(a_{1,n},\ldots,a_{n,n})^T$. Hence, in total, there are at most 
$H^{2n}\#\cD_{n-1}(\Q;H,0)$ possibilities for $a_{1,n}$ and the bottom $n-1$ rows of $A$. We now fix $a_{1,n}$ and the bottom $n-1$ rows, and we need to count only the number of possibilities for  $(a_{1,1},\ldots,a_{1,n-1})$.

To count the number of possible tuples $(a_{1,1},\ldots,a_{1,n-1})$, we look at the equation~\eqref{eqn:Lap2}, recalling that we assume that $Q_1,\ldots,Q_k\ne 0$ for $1\le k\le n-1$. If $k=1$, then $a_{1,1}$ is uniquely defined and we can choose $(a_{1,2},\ldots,a_{1,n-1})$ in $H^{2(n-2)}$ ways. Hence, in this case the number of possible matrices $A$ is at most
\begin{equation}
\label{eq:k=1}\begin{split}
 \widetilde \cD_n(\Q;H,1,\delta)  & \ll H^{2n}\#\cD_{n-1}(\Q;H,0) \cdot H^{2(n-2)}\\
 & =H^{4n-4}\#\cD_{n-1}(\Q;H,0). 
\end{split}
\end{equation}

We assume now that $2 \le k<n$  and apply Lemma~\ref{lem:aixi-F} to the equation~\eqref{eqn:Lap2}, from where we get that there are at most $H^{k+o(1)}$ possible values of $a_{1,1},a_{1,2},\ldots, a_{1,k}$. Using the trivial bound $H^{2(n-1-k)}$ for the rest of entries $a_{1,k+1},\ldots, a_{1,n-1}$, the number of possible matrices $A$ in this case is at most
\begin{equation}
\begin{split}
\label{eq:k>=2}
 \widetilde \cD_n(\Q;H,k,\delta)&\le H^{2n}\#\cD_{n-1}(\Q;H,0)   \cdot H^{k+o(1)}\cdot H^{2(n-1-k)}\\
& =    H^{4n-k-2+o(1)}\#\cD_{n-1}(\Q;H,0)  \\
&\le H^{4n-5+o(1)}\#\cD_{n-1}(\Q;H,0).
\end{split} 
\end{equation}

We then put the bounds~\eqref{eq:det 0} ,~\eqref{eq:k=n}, \eqref{eq:k=1}   and~\eqref{eq:k>=2} together to get
\begin{align*}
    \#\cD_n\(\Q;H,\delta\) &\le \max\left\{H^{2n^2-n+o(1)}, H^{4n-4}\#\cD_{n-1}(\Q;H,0) \right\}\\
    &\le \max\left\{H^{2n^2-n+o(1)}, H^{4n-4+2(n-1)^2-(n-1)} \right\}\\
    &=H^{2n^2-n+o(1)}.
\end{align*}
This completes the proof of Theorem~\ref{thm:cdn}.

\section{Proof of Theorem~\ref{thm:pnhf}}
\subsection{The case $n=2$} Let 
\[A= \begin{pmatrix}
	a_1/b_1 & a_2/b_2 \\ a_3/b_3 & a_4/b_4
\end{pmatrix}\in \cP_n(\Q;H,f)
\] 
be a $2\times 2$ matrix with characteristic polynomial 
\[
f(X)=X^2- \frac{r_1}{s_1}X+\frac{r_0}{s_0}
\]
 and $\gcd(r_1,s_1)=\gcd(r_0,s_0)=\gcd(a_i,b_i)=1$ for $i=1,2,3,4$. By using the characteristic polynomial formula and equating the polynomial coefficients, we have the following system of equations:
\begin{equation}
\label{eq:charpolx}	
\frac{a_1}{b_1}+\frac{a_4}{b_4}=\frac{r_1}{s_1} \mand 
\frac{a_1a_4}{b_1b_4}-\frac{a_2a_3}{b_2b_3}=\frac{r_0}{s_0}.
\end{equation}

It is convenient to denote $\ss = s_0s_1$. 

Substituting $a_4/b_4$ from the first equation in the second one, we have that
\begin{equation}\label{eq:rational-quadratic-charpol}
\frac{r_1a_1}{s_1b_1} - \frac{a_1^2}{b_1^2} - \frac{a_2a_3}{b_2b_3} = \frac{r_0}{s_0}.
\end{equation}
Clearing the denominators yields
\begin{equation}\label{eq:charpol}
r_1s_0a_1b_1b_2b_3= \ss b_1^2a_2a_3 + \ss a_1^2b_2b_3 + r_0s_1b_1^2b_2b_3.
\end{equation}

By reducing~\eqref{eq:charpol} modulo $b_2b_3$, we see that $b_2b_3 \mid \ss  b_1^2a_2a_3$. Thus for each such matrix, there exist auxiliary parameters $c_2,c_3, d_2,d_3$ such that:
\begin{itemize}
    \item $b_2 = c_2d_2$ and $b_3 = c_3d_3$;
    \item $c_2c_3\mid \ss  b_1^2$; 
    \item $d_3 \mid a_2$ and $d_2\mid a_3$.
\end{itemize}

Note that the parametrisation of $c_2$, $c_3$, $d_2$, and $d_3$ may not be uniquely defined, but this non-uniqueness may only increase the size of the final  bound. 

Next, we count the  sextuples $(a_1,a_2,a_3,b_1,b_2,b_3) \in [-H,H]^6$  for which there are the corresponding $c_2$, $c_3$, $d_2$, and $d_3$ in prescribed dyadic intervals. 

 More precisely, we fix some $C_2$, $C_3$, $D_2$, and $D_3$ with  
 \[
 1 \le C_2, C_3, D_2, D_3 \le H \mand C_2C_3D_2D_3 \le H^2,\] 
 and
 count the number $\fP(C_2,C_3,D_2, D_3)$ of elements of $\cP_2(\mathbb Q;H,f)$ with the additional constraints 
\begin{equation}
\label{eq:d2d3}
 c_i \in [C_i,2C_i], \qquad d_i \in [D_i, 2 D_i],   \qquad i =2,3. 
\end{equation} 
 
Clearly 
\begin{equation}
\label{eq:PPD2D3}
\# \cP_2(\mathbb Q;H,f) \ll  (\log H)^4 \max_{1 \le C_2,C_3,D_2,D_3\le H} \fP(C_2,C_3,D_2, D_3),
\end{equation}
so we may proceed by bounding $\fP(C_2,C_3,D_2,D_3)$.

We begin by fixing $c_2$ and $c_3$; we now show that, perhaps quite surprisingly, there are at most $\sqrt{C_2C_3}H^{o(1)}$ ways to do so.

\begin{lemma}\label{lem:SqrtBoundC2C3}
For any $C_2,C_3$ with $1 \le C_2,C_3 \le H$, there are at most $\sqrt{C_2C_3}H^{o(1)}$ possibilities for the pair $\(c_2, c_3\)$ satisfying~\eqref{eq:d2d3} and consistent with a solution to~\eqref{eq:charpolx}.
\end{lemma}

\begin{proof} 
By definition, for any choice of parameters $c_2,c_3,b_1$, we have $c_2c_3\mid \ss b_1^2$.

 We can parametrise further by writing $c_2 = c_{2,s}g_2$ and $c_3 = c_{3,s} g_3$, where 
\[c_{2,s}c_{3,s}\mid \ss, \mand g_2g_3 \mid b_1^2.
\] 

Furthermore, let $c$ denote the smallest positive integer such that $g_2g_3\mid  c^2$. Note that the minimality of $c$ immediately implies that 
\[
c \mid b_1 \mand c \mid g_2g_3.
\]
In particular, $c \le g_2g_3 \ll C_2 C_3$.

Begin by fixing $c_{2,s}$ and $c_{3,s}$; since these are both factors of $\ss $, 
there are at most 
\[
\fC_{2,3} = H^{o(1)}
\] 
choices for each. It may be the case that $g_2$ and $g_3$ are not relatively prime to $\ss $; write $g_2 = g_{2,s} h_2$ and $g_3 = g_{3,s}h_3$, where we can assume that $g_{2,s}$ and $g_{3,s}$ are supported only on prime divisors of 
$\ss $  and that $\gcd(h_2h_3,\ss ) = 1$. 

Now fix $g_{2,s}$ and $g_{3,s}$. Since these are supported only on prime divisors of 
$\ss $, so, by Lemma~\ref{lem: Sunits},   there are at most 
\[
\fG_{2,3}= H^{o(1)}
\] 
 choices for them. 

Similarly we write $c =   d h$ where $ d \mid g_{2,s}g_{3,s}$ and $\gcd(h,\ss ) = 1$; once again, there are at most 
\[
\fD = H^{o(1)}
\] 
choices for $d$, so we can fix one such choice. Note that by the above restriction 
$\gcd(h_2h_3,\ss ) = 1$, we have $h_2h_3 \mid h^2$. 

It remains to bound the number of possibilities for $h_2$ and $h_3$ by $\sqrt{C_2C_3}$. 

We now consider the equation~\eqref{eq:rational-quadratic-charpol} and carefully clear denominators. Multiplying both sides by $\ss g_2g_3 b_1/c$, we get that
\begin{equation}
\label{eq:4-Terms}
\frac{a_1r_1s_0g_2g_3}{c} - \frac{a_1^2\ss g_2g_3}{cb_1} - \frac{a_2a_3\ss g_2g_3b_1}{b_2b_3c} = \frac{r_0s_1g_2g_3 b_1}{c}.
\end{equation}
Since $c \mid g_2g_3 \mid b_1$,  we obviously have 
\[
\frac{a_1r_1s_0g_2g_3}{c} ,\,  \frac{ r_0s_1g_2g_3 b_1}{c}  \in \Z.
\] 

Finally, upon expanding $b_2b_3=c_2d_2c_3d_3$ and $c_2c_3 = c_{2,s}g_2c_{3,s}g_3$, the third term can be rewritten as
\[
\frac{a_2a_3\ss g_2g_3b_1}{b_2b_3c} = \left(\frac{a_2}{d_3}\right)\left(\frac{a_3}{d_2}\right) \left(\frac{\ss }{c_{2,s}c_{3,s}}\right)\left(\frac{b_1}{c}\right) \in \Z,
\]
Thus three terms in the equations~\eqref{eq:4-Terms} are integers, 
and hence so is the remaining  term
\[
\frac{a_1^2\ss g_2g_3}{cb_1} \in \Z.
\] 
By our assumption we have $\gcd(a_1,b_1) =1$, and thus  $\gcd(a_1,c)=1$, so in turn we have 
\[
\frac{\ss g_2g_3}{cb_1} \in \Z.
\] We now rewrite
\[
\frac{\ss g_2g_3}{cb_1} = \frac{\ss g_{2,s}g_{3,s} \cdot h_2h_3}{d ^2 h^2  \(b_1/c\)} \in \Z.
\]

Since $\gcd(h, \ss g_{2,s}g_{3,s})=1$, we thus have $h^2 \mid h_2h_3$. But we also have $h_2h_3 \mid h^2$, so $h^2 = h_2h_3$. In particular, the product $h_2h_3$ must always be a perfect square. Since $h_2h_3 \le c_2c_3 \ll C_2C_3$ and is a perfect square, there are 
$O\(\sqrt{C_2C_3}\)$ choices for $h_2h_3$, and thus by Lemma~\ref{lem:div} there are at most
 \[
 \fH_{2,s} \le \sqrt{C_2C_3} H^{o(1)} 
 \] 
 possibilities for $h_2$ and $h_3$.

A choice of $h_2$ and $h_3$, along with the previous choices of $g_{2,s}, g_{3,s}, c_{2,s}$, and $c_{3,s}$, determine $c_2 = c_{2,s}g_{2,s}h_2$ and $c_3 = c_{3,s}g_{3,s}h_3$. Thus, there are altogether at most  
\[
\fC_{2,3} \fD  \fG \fH_{2,3} \le \sqrt{C_2C_3} H^{o(1)}
\]  possibilities for $(c_2,c_3)$, as desired.
\end{proof}

Now that we see from  Lemma~\ref{lem:SqrtBoundC2C3}   that  $\(c_2, c_3\)$ can be fixed in at most $\sqrt{C_2C_3}H^{o(1)}$ ways, we fix positive integers $d_2$ and $d_3$ that satisfy~\eqref{eq:d2d3}. 
Note that $b_2=c_2d_2$, $b_3 = c_3 d_3$, and $c$ are now determined as well. 
    
We now look at the equation~\eqref{eq:charpol} as a quadratic equation with respect to 
$a_1$, which must have an integer root. This means that its discriminant 
\[
\Delta = r_1^2s_0^2b_1^2b_2^2b_3^2 - 4\ss b_2b_3\(\ss b_1^2a_2a_3 + r_0s_1b_1^2b_2b_3\) 
\]
must be a perfect square.
Recalling that $c\mid b_1$ and writing
\[
\Delta  = \left(\frac{b_1}{c}\right)^2b_2^2b_3^2 R
\]
with 
\[ R = r_1^2 s_0^2 c^2 -  4\frac{s_0^2s_1^2c^2}{c_2c_3} \frac{a_2}{d_3}\frac{a_3}{d_2} - 4r_0 \ss s_1 c^2,  
\]
we see that $R$ must also be a perfect integer square. 

We now observe that $R$ belongs to the arithmetic progression
\[
R  \equiv A \pmod{q}, 
\]   
where 
\[
A= r_1^2 s_0^2 c^2- 4r_0 \ss s_1 c^2 \mand q =  4\frac{s_0^2s_1^2c^2}{c_2c_3}
\]
 are fixed. 
 Moreover, we see that $R = A+kq$ for some
 \[
 |k| \le H^2/(d_2 d_3),
 \]
and hence, by  Lemma~\ref{lem:Rudin}, $R$ can take at most $H^{1+o(1)}(d_2 d_3)^{-1/2}$ values.
When $R$ is fixed the product $a_2a_3$ is uniquely defined for given $d_2$ and $d_3$.
Hence, by Lemma~\ref{lem:div} we have  at most $H^{1+o(1)}(d_2 d_3)^{-1/2}$
choices for $(a_2,a_3)$. 
Summing over all admissible $d_2$ and $d_3$ we obtain 
\begin{equation}
\label{eq:a2a3}
 \sum_{D_2\le d_2\le 2D_2}  \sum_{D_3\le d_3\le 2D_3} 
\frac{H^{1+o(1)}}{(d_2 d_3)^{1/2} }  = H^{1+o(1)}(D_2 D_3)^{1/2}
\end{equation}
choices for $(a_2,a_3)$.

When $(a_2,a_3,b_2,b_3)$ are fixed, then $a_1/b_1$ is the root of a rational quadratic polynomial, from~\eqref{eq:rational-quadratic-charpol}, so $a_1$ and $b_1$ are uniquely defined (up to a factor of $2$). Then $a_4$ and $b_4$ are uniquely defined from~\eqref{eq:charpolx}. Hence from Lemma~\ref{lem:SqrtBoundC2C3} and~\eqref{eq:a2a3} we derive 
the bound
\begin{equation}
\label{eq:BoundWithC2C3D2D3}
 \fP(C_2,C_3,D_2, D_3,) \le (C_2C_3)^{1/2}H^{1+o(1)}(D_2 D_3)^{1/2}.
\end{equation}
Recall that $C_2C_3D_2D_3 \le H^2$; together with~\eqref{eq:BoundWithC2C3D2D3} we have 
\begin{equation*}
\fP(C_2,C_3,D_2,D_3) \le H^{2+o(1)},
\end{equation*}
which together with~\eqref{eq:PPD2D3} concludes the proof.

\subsection{The case $n\ge 3$}
We count matrices $A \in \cP_n(\Q;f)$, and thus, in particular,  matrices $A=(a_{i,j}/b_{i,j})_{1\le i,j\le n}$, $a_{i,j}/b_{i,j}\in\cF(H)$, for which $\Tr A$ and $\Tr A^2$ are fixed. Indeed, $\Tr A$ is defined by the coefficient of $X^{n-1}$ in $f$ and $\Tr A^2$ is defined by the coefficient of $X^{n-2}$ in $f$ (via the Newton identities). 

Let $\Tr A=r_1/s_1$ be fixed, with $\gcd(r_1,s_1)=1$ and $s_1\ne 0$, and after clearing denominators, we obtain the linear equation
\[
s_1\frac{a_{1,1}}{b_{1,1}}+\cdots+s_1\frac{a_{n,n}}{b_{n,n}}=r_1
\]
in variables $a_{i,i}/b_{i,i}\in\cF(H)$, $i=1,\ldots,n$, with fixed $r_1,s_1\in\Z$, of size
\[
|r_1|, |s_1|\le H^{O(1)}.
\]
Applying Lemma~\ref{lem:aixi-F}, we obtain that there are at most
\begin{equation}
\label{eq:tr bound}
H^{n+o(1)}
\end{equation}
solutions $(a_{1,1}/b_{1,1},\ldots,a_{n,n}/b_{n,n})\in\cF(H)^n$. We fix one such choice.

We turn now our attention to $\Tr A^2$, which is fixed. Since
\[
\Tr A^2=\sum_{i=1}^n \sum_{j=1}^n \frac{a_{i,j}}{b_{i,j}}\cdot \frac{a_{j,i}}{b_{j,i}}
\]
and $a_{1,1}/b_{1,1},\ldots,a_{n,n}/b_{n,n}$ and $\Tr A^2$ are fixed, after clearing denominators, we obtain the equation
\begin{equation}
\label{eq:tr2 eq}
\sum_{1\le i<j\le n}s_2\frac{a_{i,j}}{b_{i,j}}\cdot \frac{a_{j,i}}{b_{j,i}}=r_2,
\end{equation}
with fixed $r_2,s_2\in\Z$, $s_2\ne 0$ (since it divides the product of the denominators of $\Tr A^2$ and $a_{1,1}/b_{1,1},\ldots,a_{n,n}/b_{n,n}$), of size
\[
|r_2|,|s_2|\le H^{O(1)}.
\]
For $1\le i<j\le n$, denote
\begin{equation}
\label{eq:abij}
A_{i,j}=a_{i,j}a_{j,i},\quad B_{i,j}=b_{i,j}b_{j,i},\quad d_{i,j}=\gcd(A_{i,j},B_{i,j}).
\end{equation}
Therefore, we can write
\[
A_{i,j}=d_{i,j}C_{i,j},\quad B_{i,j}=d_{i,j}D_{i,j},\quad \gcd(C_{i,j},D_{i,j})=1.
\]
Equation~\eqref{eq:tr2 eq} becomes
\begin{equation}
\label{eq:tr2 eq1}
\sum_{1\le i<j\le n}s_2\frac{C_{i,j}}{D_{i,j}}=r_2,
\end{equation}
in $n(n-1)/2$ variables $C_{i,j}/D_{i,j}\in\cF(H^2/d_{i,j})$, $1\le i<j\le n$. For fixed $d_{i,j}$, $1\le i<j\le n$, applying again Lemma~\ref{lem:aixi-F}, we obtain that there are at most
\[
H^{2\cdot \frac{n(n-1)}{2}+o(1)}\prod_{1\le i<j\le n}1/d_{i,j}=H^{n^2-n+o(1)}\prod_{1\le i<j\le n}1/d_{i,j}
\]
solutions $C_{i,j}/D_{i,j}\in\cF(H^2/d_{i,j})$, $1\le i<j\le n$, and thus the same number of solutions $A_{i,j}/B_{i,j}\in\cF(H^2)$, $1\le i<j\le n$.
 
 Since, by~\eqref{eq:abij}, we have
\[
 a_{i,j},a_{j,i}\mid A_{i,j}\mand b_{i,j},b_{j,i}\mid B_{i,j},
\]
 applying Lemma~\ref{lem:div}, we conclude that for each choice of $A_{i,j}$ and $B_{i,j}$, there are at most $H^{o(1)}$ possibilities for the entries $a_{i,j}/b_{i,j},a_{j,i}/b_{j,i}\in\cF(H)$. Therefore, for fixed $d_{i,j}$, $1\le i<j\le n$, the number of entries $a_{i,j}/b_{i,j}$, $i\ne j$, is at most 
\[
H^{n^2-n+o(1)}\prod_{1\le i<j\le n}1/d_{i,j}.
\]
Therefore, the number of entries $a_{i,j}/b_{i,j}$, $i\ne j$, is at most 
 \begin{equation}
\label{eq:tr2 bound}
\begin{split}
\ssum_{1\le d_{1,2}, \ldots, d_{n-1,n}\le H^2}  & H^{n^2-n+o(1)}\prod_{1\le i<j\le n}1/d_{i,j}\\
& \qquad  \qquad = H^{n^2-n+o(1)}\prod_{1\le i<j\le n} \sum_{1\le d_{i,j}\le H^2} 1/d_{i,j}\\
&\qquad\qquad \le H^{n^2-n+o(1)}(\log H)^{(n^2-n)/2}\\
&\qquad\qquad\qquad =H^{n^2-n+o(1)}, 
\end{split}
\end{equation}
where the sum on the left hand side is over all $n(n-1)$ dimensional integer vectors 
$ \(d_{i,j}\)_{1\le i<j\le n}$ of size at most $H$. 

 Putting~\eqref{eq:tr bound} and~\eqref{eq:tr2 bound} together, we obtain at most
 $H^{n^2+o(1)}$
 matrices $A\in\cP_n(\Q;H)$, which concludes the proof of Theorem~\ref{thm:pnhf}.
 
  \section{Proof of Theorem~\ref{thm:rank 1/Z}}

 \subsection{The case $r=1$}  Let 
 \[A=(1/a_{i,j})_{\substack{1\le i \le m\\1\le j \le n}}\in \cM_{m,n}(\Z^{-1},H)
 \]
 be a matrix with rank 1.  Since $\rank A=1$, all $2\times 2$ submatrices of $A$ are singular.
 
 We first fix the entries $a_{i,i}$ for $i=1,\dotsc,m$, and consider the $2\times 2$ submatrices of $A$ of the form 
 \[
     \begin{pmatrix}
         1/a_{i_1,  i_1}&1/a_{i_1,i_2}\\
         1/a_{i_2,i_1}&1/a_{i_2,i_2}
     \end{pmatrix}
\]
 for $1\le i_1<i_2\le m$. Since this matrix is singular, we have that, from calculating its determinant, 
 \[
     a_{i_1,i_1}a_{i_2,i_2}=a_{i_1,i_2}a_{i_2,i_1}.
\]
 The left-hand side of this equation is fixed. Hence, by Lemma~\ref{lem:div}, there are $\tau(a_{i_1,i_1}a_{i_2,i_2}) = H^{o(1)}$ choices for $(a_{i_1,i_2},a_{i_2,i_1})$. Repeating this argument for all choices of $1\le i_1 <i_2\le m$, we see that there are at most $H^{m+o(1)}$ possible choices for the entries $a_{i,j}$, with $1\le i,j \le m$.

 Now, we fix such a choice of $a_{i,j}$, for $1\le i,j \le m$. We further fix the entries $a_{1,j}$ for $m+1\le j \le n$ in $H^{n-m}$ ways. We now consider the submatrix
 \[
     \begin{pmatrix}
         a_{1,1}&a_{1,j}\\
         a_{i, 1}&a_{i,j}
     \end{pmatrix} 
\]
 for  $1\le i \le m$ and $m+1\le j \le n$.  
 This matrix is singular. Then, since the entries $a_{1,1}$, $a_{1,j}$ and $a_{i,1}$ are already fixed we see that $a_{i,j}$ is uniquely defined.

 From our argument, we see that for fixed entries $a_{i,i}$ and $a_{1,j}$ of $A$, with $1\le i\le m$ and $m + 1\le j \le n$, there are at most $H^{o(1)}$ possible choices for the other entries of $A$. Hence, there are at most $H^{n+o(1)}$ possible matrices $A \in \cM_{m,n}\(\Z^{-1};H\)$ such that $\rank A=1$. This completes our proof in this case.

\subsection{The case $r=2$} We first remark that in the proof we use a bound 
of Theorem~\ref{thm:cdnE} on  $\#\cD_2\(\Z^{-1};H,0\)$ and $\#\cD_3\(\Z^{-1};H,0\)$. However one easily checks
that this result is independent of our bound on $\#\cL_{m,n, 2}$. So there is no circular 
reasoning here. 

 Let 
 \[A=(1/a_{i,j})_{\substack{1\le i \le m\\1\le j \le n}}\in \cL_{m,n,2}\(\Z^{-1};H\).
 \]
  Since $\rank A=2$, all $3\times 3$ submatrices of $A$ are singular. However, there exists a $2\times 2$ submatrix $A$ that is not singular. Without loss of generality, we may suppose that the submatrix \begin{equation}\label{eq:matrix}
    \begin{pmatrix}
        1/a_{1,1}& 1/a_{1,2}\\
        1/a_{2,1}& 1/a_{2,2}
    \end{pmatrix}
\end{equation}  
is not singular.

We note that if we fix the first two rows and columns of $A$, for all $3\le i \le m$, $m+1\le j\le n$, the matrix \[
    \begin{pmatrix}
        1/a_{1,1}& 1/a_{1,2}&  1/a_{1,j}\\
        1/a_{2,1}& 1/a_{2,2}& 1/ a_{2,j}\\
        1/a_{i,1}& 1/a_{i,2}&  1/a_{i,j}
    \end{pmatrix}
\] 
has determinant zero. Note that in the above matrix all elements except for $1/a_{i,j} \in \cF(H)$ 
are fixed. Hence, there exists at most one possible choice of $1/a_{i,j} \in \cF(H)$ for $3\le i \le m$,  $3\le j\le n$. 

Note that we only use the above argument with $i,j\ge 4$ and examine more carefully the 
number of possible choices for the first three rows and columns of $A$, which 
by the above observation gives us a bound on $\# \cL_{m,n,2}\(\Z^{-1};H\)$. 

Next, we divide the counting of $A \in \cL_{m,n,2}\(\Z^{-1};H\)$ based on the structure of the $2\times 2$ submatrices  
\[
    \begin{pmatrix}
        1/a_{1,k}&1/a_{1,j}\\
        1/a_{2,k}&1/a_{2,j}
    \end{pmatrix} \mand  \begin{pmatrix}
        1/a_{k,1}&1/a_{k,2}\\
        1/a_{i,1}&1/a_{i,2}
 \end{pmatrix}
\]
for $k=1,2$ and $3\le i \le m$, $3\le j\le n$. 

We say that a matrix $A \in \cL_{m,n,2}\(\Z^{-1};H\)$  is \textit{row-good} if there exists an index $3\le j \le n$ such that 
\[
    \begin{pmatrix}
        1/a_{1,1}&1/a_{1,j}\\
        1/a_{2,1}&1/a_{2,j}
    \end{pmatrix} \mand  \begin{pmatrix}
        1/a_{1,2}&1/a_{1,j}\\
        1/a_{2,2}&1/a_{2,j}
    \end{pmatrix}
\]
 are not singular.  Otherwise we says that the matrix $A$ is \textit{row-bad}.

Similarly, we say that $A \in \cL_{m,n,2}\(\Z^{-1};H\)$ is \textit{column-good} or \textit{column-bad} depending on whether  there exists an index $3\le i\le m$ such that both 
\[
    \begin{pmatrix}
        1/a_{1,1}&1/a_{1,2}\\
        1/a_{i,1}&1/a_{i,2}
    \end{pmatrix} \mand  \begin{pmatrix}
        1/a_{2,1}&1/a_{2,2}\\
        1/a_{i,1}&1/a_{i,2}
    \end{pmatrix}
\]
are not singular.

We first prove that if a matrix is row-bad, then  for each $j$, $3\le j\le n$, exactly one of the submatrices 
\[
    \begin{pmatrix}
        1/a_{1,1}&1/a_{1,j}\\
        1/a_{2,1}&1/a_{2,j}
    \end{pmatrix} \mand  \begin{pmatrix}
        1/a_{1,2}&1/a_{1,j}\\
        1/a_{2,2}&1/a_{2,j}
    \end{pmatrix}
\]
is singular. Assume that both submatrices are singular, then by the determinant formula we have
 \[
        a_{1,1}a_{2,j}=a_{1,j}a_{2,1}\mand
        a_{1,2}a_{2,j}=a_{1,j}a_{2,2},
\]
and thus
\[
        a_{1,1}a_{2,2}=a_{1,2}a_{2,1},
\]
which contradicts the nonsingularity of~\eqref{eq:matrix}. This completes the proof. The same argument can be used to prove that in a column-bad matrix, for each $3\le i\le m$, exactly one of the submatrices
 \[
    \begin{pmatrix}
        1/a_{1,1}&1/a_{1,2}\\
        1/a_{i,1}&1/a_{i,2}
    \end{pmatrix} \mand  \begin{pmatrix}
        1/a_{2,1}&1/a_{2,2}\\
        1/a_{i,1}&1/a_{i,2}
    \end{pmatrix}
\]
is singular.

We now divide the cases based on whether the matrix $A$ is row-good or column-good. Denote $L_{GG}$, $L_{GB}$, $L_{BG}$ and $L_{BB}$ as the number of matrices $A \in \cL_{m,n,2}\(\Z^{-1};H\)$ which are row-good and column-good, row-good and column-bad, row-bad and column-good, and row-bad and column-bad, respectively.

We first bound $L_{GG}$. 
Without loss of generality (by permutations), we may assume that the matrices 
\[
    \begin{pmatrix}
        1/a_{1,k}&1/a_{1,3}\\
        1/a_{2,k}&1/a_{2,3}
    \end{pmatrix} \mand  \begin{pmatrix}
        1/a_{k,1}&1/a_{k,2}\\
        1/a_{3,1}&1/a_{3,2}
    \end{pmatrix}
\] 
are not singular, for $k=1,2$. We now consider the  submatrix 
\begin{equation}\label{eq:3x3}\begin{pmatrix}
        1/a_{1,1}& 1/a_{1,2}&  1/a_{1,3}\\
        1/a_{2,1}& 1/a_{2,2}&  1/a_{2,3}\\
        1/a_{3,1}& 1/a_{3,2}&  1/a_{3,3}
    \end{pmatrix}. \end{equation}

By Theorem~\ref{thm:cdnE}, since this matrix is singular, there are at most $H^{7+o(1)}$ possible choices for the entries of this submatrix. We fix one such choice.

Next, we consider the matrix 
\begin{equation}\label{eq:12i}
\begin{pmatrix}
        1/a_{1,1}& 1/a_{1,2}&  1/a_{1,3}\\
        1/a_{2,1}& 1/a_{2,2}& 1/a_{2,3}\\
        1/a_{i,1}& 1/a_{i,2}&  1/a_{i,3}
    \end{pmatrix}
 \end{equation} 
for $4\le i\le m$. We note that the first two rows of this matrix are already fixed. Since this matrix is singular, by Laplace expansion, we have an equation of the form 
\begin{equation}\label{eq:Q1Q2Q3}
\frac{Q_1}{a_{i,1}}+
\frac{Q_2}{a_{i,2}}+
\frac{Q_3}{a_{i,3}}=0,
\end{equation} 
where all of $Q_1$, $Q_2$, and $Q_3$ are not zero since $A$ is row-good.
Using Lemma~\ref{lem:aixi-Z n=2,3}, there are at most $H^{1+o(1)}$ possible choices of $(1/a_{i,1}$, $1/a_{i,2},1/a_{i,3})$. Therefore, there are at most  $H^{m-3+o(1)}$
 possible choices for $(1/a_{i,1}$, $1/a_{i,2}$,   $1/a_{i,3})$ for all $4\le i \le m$. 

The same argument shows that there are  at most  $H^{n-3+o(1)}$ possible choices for $(1/a_{1,j}$, $1/a_{2,j}$,   $1/a_{3,j})$ for all $4\le j \le n$. 

Multiplying all bounds, we see that there are at most 
\begin{equation}\label{eqn:g1}
L_{GG} \le H^{7+o(1)}\cdot H^{m-3+o(1)}\cdot H^{n-3+o(1)} = H^{m+n+1+o(1)}.
\end{equation} 
matrices $A \in \cL_{m,n,2}\(\Z^{-1};H\)$ which are both row-good and column-good.
.
Next, we  bound $L_{GB}$, the number of row-good and column-bad 
matrices $A \in \cL_{m,n,2}\(\Z^{-1};H\)$. We first bound the number of possible entries in the matrix~\eqref{eq:3x3}.

Since $A$ is column-bad, for each $3\le i 
\le m$, exactly one of the submatrices
\[
    \begin{pmatrix}
        1/a_{1,1}&1/a_{1,2}\\
        1/a_{i,1}&1/a_{i,2}
    \end{pmatrix} \mand  \begin{pmatrix}
        1/a_{2,1}&1/a_{2,2}\\
        1/a_{i,1}&1/a_{i,2}
    \end{pmatrix}
\]
is singular. Suppose that the matrix 
\[
\begin{pmatrix}
        1/a_{2,1}&   1/a_{2,2}\\
        1/a_{3,1}&   1/a_{3,2}
\end{pmatrix}
 \]
 is singular. Using Theorem~\ref{thm:cdnE}, we see that there are at most $H^{2+o(1)}$ ways of choosing the entries of this submatrix. Next, there are $O(H^2)$ ways of choosing $(1/a_{1,1},1/a_{1,2})$. This concludes the argument of the bound for the first two columns of the matrices~\eqref{eq:3x3}.

    It remains to bound the number of possible entries for the last column of~\eqref{eq:3x3}. To do this, we first note that, since $A$ is row-good, we may assume that both submatrices 
 \[
    \begin{pmatrix}
        1/a_{1,1}&1/a_{1,3}\\
        1/a_{2,1}&1/a_{2,3}
    \end{pmatrix} \mand  \begin{pmatrix}
        1/a_{1,2}&1/a_{1,3}\\
        1/a_{2,2}&1/a_{2,3}
    \end{pmatrix}
\] 
are nonsingular. By applying Laplace expansion to its last column, we have an equation of the form 
 \[
        \frac{Q_2}{a_{2,3}}+\frac{Q_3}{a_{3,3}}=0,
\]
where $Q_2,Q_3\neq 0$. By Lemma~\ref{lem:aixi-Z n=2,3} we see that this equation has at most $H^{1+o(1)}$ solutions $(1/a_{2,3},1/a_{3,3}) \in \cF(H)^2$. Since there are $O(H)$ possible entries for $1/a_{1,3}$ such that~\eqref{eq:3x3} is row-good, we conclude that there are at most $H^{6+o(1)}$ possible  choices of the matrix~\eqref{eq:3x3}.

    We now proceed with bounding the number of possible entries of $(1/a_{1,j},1/a_{2,j},1/a_{3,j})$ for $4\le j\le n$. To do this, we consider the matrix
 \begin{align*}\begin{pmatrix}
        1/a_{1,1}& 1/a_{1,2}&  1/a_{1,j}\\
        1/a_{2,1}& 1/a_{2,2}&  1/a_{2,j}\\
        1/a_{3,1}& 1/a_{3,2}&  1/a_{3,j}
    \end{pmatrix} \end{align*}
 and repeating our previous argument on bounding the number of possible choices of $(1/a_{1,3},1/a_{2,3},1/a_{3,3})$, we have that there are at most $H^{2+o(1)}$ possible choices of $(1/a_{1,j},1/a_{2,j},1/a_{3,j})$, for each $4\le j\le n$. Hence, 
 the total number of choices for $n-3$ triples  $(1/a_{1,j},1/a_{2,j},1/a_{3,j})$,  $4\le j\le n$, is  at most $ H^{2(n-3)+o(1)}$.

Finally, we bound the number of possible choices for the triples $(1/a_{i,1},1/a_{i,2},1/a_{i,3})$ for $4\le i\le m$.  We now consider the matrix~\eqref{eq:12i} 
for $4\le i \le m$ and get  the equation~\eqref{eq:Q1Q2Q3}, since $A$ is row-good. Thus, following the same argument used in bounding $L_{GG}$, we see that there are at most $H^{m-3+o(1)}$ possibilities of choosing  $(1/a_{i,1}$, $1/a_{i,2}$,   $1/a_{i,3})$, for all $4\le i \le n$.

Multiplying the above bounds we see that there are at most \begin{equation}\label{eqn:g2}
L_{GB} \le H^{6+o(1)} \cdot H^{2(n-3)+o(1)}\cdot  H^{m-3+o(1)} = H^{m+2n-3+o(1)}
\end{equation} 
matrices in $\cL_{m,n,2}\(\Z^{-1};H\)$ which are row-good and column-bad.

Next, we bound $L_{BG}$. We follow the exact same arguments as our bounds on $L_{GB}$, the only difference is that we swap the ``row" and ``column" in this case. Using this, we have 
 \begin{equation}\label{eqn:g3}
L_{BG} \le H^{6+o(1)} \cdot H^{2(m-3)+o(1)}\cdot  H^{n-3+o(1)} = H^{2m+n-3+o(1)}.
\end{equation} matrices in $\cL_{m,n,2}\(\Z^{-1};H\)$ which are row-bad and column-good.

Finally, we  bound $L_{BB}$. To do this, we first fix the matrix 
\[
   \begin{pmatrix}
        1/a_{1,1}&1/a_{1,2}\\
1/a_{2,1}&1/a_{2,2}
\end{pmatrix}
\]
in $O(H^4)$ ways and  consider the matrix
 \begin{equation}\label{eq:3x3pq}
   \begin{pmatrix}
        1/a_{1,1}&1/a_{1,2}&1/a_{1,j}\\
    1/a_{2,1}&1/a_{2,2}&1/a_{2,j}\\
    1/a_{i,1}&1/a_{i,2}&1/a_{i,j}
   \end{pmatrix}
\end{equation} for $3\le i\le m$, $3\le j\le n$. Without loss of generality, we may assume that the matrices 
\[
    \begin{pmatrix}
        1/a_{1,2}&1/a_{1,j}\\
1/a_{2,2}&1/a_{2,j}
    \end{pmatrix} \mand  \begin{pmatrix}
        1/a_{2,1}&1/a_{2,2}\\
    1/a_{i,1}&1/a_{i,2}
    \end{pmatrix}
\] are singular. We see that the last statement is equivalent to the system of equations 
\[
    a_{1,2}a_{2,j}=a_{2,2}a_{1,j}\mand
    a_{2,1}a_{i,2}=a_{2,2}a_{i,1}.
\]
We fix $a_{1,j}$ and $a_{i,1}$ in $O(H)$ ways each, and recalling that  $a_{2,2}$ has already been fixed, we see that from Lemma~\ref{lem:div}, there are $H^{o(1)}$ possibilities for $a_{2,j}$ and $a_{i,2}$ each. And since the matrix in~\eqref{eq:3x3pq} is singular, there is at most one possibility for the entry $1/a_{i,j}$ for each $3\le i \le m$, $3\le j \le n$. 

Repeating this argument for each $i$ and $j$ with $3\le i\le m$, $3\le j\le n$, we see that there are at most \begin{equation}\label{eqn:g4}
L_{BB} \le H^{4}\cdot H^{m-2+o(1)} \cdot H^{n-2+o(1)}=H^{m+n+o(1)}
\end{equation} 
matrices in  $\cL_{m,n,2}\(\Z^{-1};H\)$ which are row-bad and column-bad.

Putting~\eqref{eqn:g1}, \eqref{eqn:g2}, \eqref{eqn:g3} and~\eqref{eqn:g4} together, we have that \begin{align*}
    \#\cL_{m,n,2}\(\Z^{-1};H\) &\ll  L_{GG}+L_{GB}+ L_{BG} + L_{BB}\\
     &\le H^{m+n+1+o(1)}+ H^{2n+m-3+o(1)}\\
     & \qquad \qquad + H^{2m+n-3+o(1)}+ H^{m+n+o(1)} \\
    &= \begin{cases}
        H^{7+o(1)},&\text{ if }(m,n)=(3,3),\\
        H^{2n+m-3+o(1)},&\text{ if }(m,n)\ne (3,3).
    \end{cases}
\end{align*} This completes the proof for the case $r=2$.

\subsection{The case $r\ge 3$} In this case, we proceed similarly with the proof of Theorem~\ref{thm:rank Q}, with the main difference being that we work with the set $\cE(H)$ of Egyptian fractions of height at most $H$, instead of the set of Farey fractions $\cF(H)$. 

Indeed, following the proof of  Theorem~\ref{thm:rank Q} line by line, and using the same notation, we estimate the number $\#\cL_{m,n, r}^*(\Z^{-1};H)$ of matrices $A \in \cM_{m,n}\(\Z^{-1};H\)$, which are of rank $r$ and such that the top left $r \times r$ minor $A_r = \(a_{i,j}\)_{1\le i,j\le r} $ is non-singular. We obtain the analogue bounds of~\eqref{eq:contrib Xr} and~\eqref{eq:contrib Yr-t} given by
\begin{equation}
\label{eq:contrib Xr E}
\cI \ll H^{r^2} 
\end{equation} 
and 
\begin{equation}
\label{eq:contrib Yr-t E}
\cJ_t \ll H^{t(m-r)},
\end{equation}
and the analogue of~\eqref{eq:lin rel} given by
\[
\frac{\rho_{1}(h)}{a_{1,j}}  + \ldots  + \frac{\rho_{r}(h)}{a_{r,j}} - \frac{1}{a_{h,j}} =0
\]
with exactly $t+1$ non-zero coefficients and $1/a_{i,j},\ldots ,1/a_{r,j},1/a_{h,j}\in \cE(H)$. To count the number of solutions to this equation, we continue following the proof of  Theorem~\ref{thm:rank Q}, where instead of Lemma~\ref{lem:aixi-F} we use Lemma~\ref{lem:aixi-Z}.

Thus, in this case we obtain
\begin{itemize}
\item  $ \cK_t \ll H\# \cE(H)^{r-1} \ll H^{r}$  if $t=1$, 
\item  $ \cK_t \le H^{1+o(1)}\# \cE(H)^{r-2} \le H^{r-1+o(1)}$  if $t=2$,  
\item   $ \cK_t \le  H^{(t+1)/2+o(1)} \#\cE(H)^{r-t}  \le  H^{r - t/2+1/2+o(1)}$,  which follows from 
Lemma~\ref{lem:aixi-Z}  if $  t \ge 3$.  
\end{itemize}

Now we are ready to derive an upper bound on $\#\cL_{m,n,r}^*(\Z^{-1};H)$ from these arguments. From our definitions of $\cI$, $\cJ_t$, and $\cK_t$, we have
\begin{equation}
\label{eq: L vs ABC E}
\#\cL_{m,n, r}^*(\Z^{-1};H)   \ll \cI   \sum_{t =1}^r \cJ_t \cK_t^{n-r}.
\end{equation} 

We now analyse the contributions of each term of the right-hand side of~\eqref{eq: L vs ABC E}.
In particular, we recall the bounds~\eqref{eq:contrib Xr E} and~\eqref{eq:contrib Yr-t E}
as well as the above bounds on $\cK_t$.

For $t =1 $ the total contribution $\#\fL_{1}$ satisfies
\begin{equation}\label{eq:t = 1}
\#\fL_{1} \ll H^{r^2} H^{m-r} \(H^{r}\)^{n-r} = H^{nr+m-r}. 
\end{equation} 

For $t =2 $ the total contribution $\#\fL_{2}$ satisfies
\begin{equation}\label{eq:t = 2}
\#\fL_{2} \le H^{r^2}  H^{2(m-r)}\(H^{r-1+o(1)}\)^{n-r} = H^{nr+2m-n-r+o(1)}. 
\end{equation} 

 For $t \ge 3$ the total contribution $\#\fL_{\ge 3}$ satisfies 
\begin{equation}\label{eq:t ge 3}
\begin{split}
\#\fL_{\ge 3}&\le H^{r^2}\sum_{t =3}^r   H^{t(m-r)} \( H^{r - t/2 + 1/2+o(1)}\)^{n-r}  \\
& = H^{nr+n/2-r/2+o(1) }\sum_{t =3}^r  H^{t(m-n/2-r/2)}\\
& = H^{nr+n/2-r/2+o(1) } \( H^{3m-3n/2-3r/2}+ H^{mr-nr/2-r^2/2}\)\\
& = H^{nr+3m-n-2r+o(1) }  +  H^{(n-r)(r+1)/2+mr+o(1) }  .
\end{split} 
\end{equation}

Hence, we can now estimate $ \#\cL_{m,n, r}\(\Z^{-1};H\)$  by substituting the bounds~\eqref{eq:t = 1}, \eqref{eq:t = 2} and~\eqref{eq:t ge 3} in the inequality 
\[
 \#\cL_{m,n, r}\(\Z^{-1};H\)  \ll \#\cL_{m,n, r}^*(\Z^{-1};H) \le \#\fL_{1}   + \#\fL_{2}   
+ \#\fL_{\ge 3} , 
\] 
and recalling that $n\ge m\ge r \ge 3$, we see that 
\begin{align*}
 \#\cL_{m,n, r}&(\Z^{-1};H)\\
  &\ll H^{nr+m-r} +    H^{nr+3m-n-2r+o(1) }  +  H^{(n-r)(r+1)/2+mr+o(1) }.
	\end{align*}
 
 We then compare these terms.  We have that
\begin{align*}
&  H^{nr+3m-n-2r} /H^{nr+m-r} = H^{2m-n-r},\\
 &H^{(n-r)(r+1)/2+mr}/H^{nr+m-r} = H^{(r-1)(2m-n-r)/2},\\
&H^{(n-r)(r+1)/2+mr}/H^{nr+3m-n-2r } = H^{(r-3)(2m-n-r)/2}.
\end{align*}

Therefore, since $r\ge 3$, we see that
\[
     \#\cL_{m,n, r}\(\Z^{-1};H\)\ll \begin{cases}
 H^{(n-r)(r+1)/2+mr+o(1)},&\text{ if }2m\ge n+r,\\
 H^{nr+m-r},&\text{ if }2m< n+r.
     \end{cases} 
\]
This completes the proof of all $r$.

 \section{Proof of Theorem~\ref{thm:cdnE}}
 \subsection{The case $n=2$}  
 \label{sec:n=2E}
 Let $\delta=r/s$ be a rational number, with integers $r$ and $s\ge 1$ satisfying $\gcd(r,s)=1$.
 Let 
 \[A= \begin{pmatrix}
 	1/x_1 & 1/x_2 \\ 1/x_3 &1/x_4
 \end{pmatrix}\in \cM_2(\Z^{-1};H)
 \] be a $2\times 2$ matrix with $\det A= \delta$ and $|x_i|\le H$ for $i=1,2,3,4$. Applying the determinant formula, we have that \begin{equation}\label{eq:detZ-1}
\frac{1}{x_1x_4}-\frac{1}{x_2x_3} = \frac{r}{s}.
\end{equation}

If $\delta =0$, then the equation~\eqref{eq:detZ-1} is equivalent to 
\[
	x_1x_4=x_2x_3.
\] 
By fixing $x_2$ and $x_3$ in $O\(H^2\)$ ways, we see that, by the bound of Lemma~\ref{lem:div}, there are $H^{o(1)}$ ways of choosing $x_1$ and $x_4$ each. This completes the proof in the case $\delta=0$.

Now we proceed to the case of $\delta\neq 0$. After clearing denominators, we may rewrite the equation~\eqref{eq:detZ-1} as \begin{equation}\label{eq:Z-1}
(rx_2x_3+s)(rx_1x_4-s)=-s^2.
 \end{equation} Now, let $rx_1x_4-s=K$.  Then, as $r$ and $s$ are fixed, we see that 
 \[
 x_1x_4 \mid (K+s)/r.
\]
Applying Lemma~\ref{lem:div}, we conclude that there are $((K+s)/r)^{o(1)}\le H^{o(1)}$ possibilities for $x_1$ and $x_4$ each, for a fixed $K$.  
 
However, since $K\mid s^2$ from the equation~\eqref{eq:Z-1}, we may conclude that that there are at most $(s^2)^{o(1)}\le  H^{o(1)}$ possible values of $(x_1,x_4)$. Applying the same argument for the expression $rx_2x_3+s$, we conclude that there are at most  $H^{o(1)}$ possible values of $(x_1,x_2,x_3,x_4)$ that satisfy the equation~\eqref{eq:detZ-1}. This completes the proof.

 \begin{rem}
 For $n=2$, Theorem~\ref{thm:pnhfE}  gives an alternative proof since 
 the characteristic polynomial of $A\in \cD_2(\Z^{-1};H,0)$ is of form $X^2-  t X$, 
 with some $t$ of the form $t=1/x_1+1/x_4$,  $|x_1|,|x_4|\le H$.  Hence, 
    \begin{align*}
        \#\cD_2(\Z^{-1};H,0) =& \sum_{\substack{t=1/x_1+1/x_4\\
        |x_1|,|x_4|\le H}}\#\cP_2(\Z^{-1};H,X^2- t X)\\
        \le& H^{2}H^{o(1)}+H^{1+o(1)}=H^{2+o(1)},
    \end{align*} where the  term $H^{1+o(1)}$ comes from $\#\cP_2\(\Z^{-1};H,X^2\)$, which appears 
    for $O(H)$ choices of $(x_1,x_4)$.  
\end{rem}
  
 \subsection{The case $n=3$}   
  Consider a $3\times 3$ matrix $A=\(1/a_{i,j}\)_{i,j=1}^3$ with entries in $\cE(H)$ and determinant $\delta$.

As in Section~\ref{laplace}, let $M_{1,j}$ be  the minor of $A$ resulted by removing the first row and $j$-th column of $A$ 	
and let $L_j$ be the least common multiple of the denominators 
of all entries of $M_{1,j}$, $j =1, 2, 3$. Then for 
 	\[
 	L = \lcm[L_1, L_2, L_3] 
 	\]
we define
 \[
  Q_0 = L\delta \mand 	Q_j =  (-1)^{1+j}   L \det M_{1,j}, \qquad j=1, 2,3.
 	\]
Therefore, we have
\[
|Q_j|\le H^{O(1)},\quad j=0,1, 2,3.
\]
Fix the last two rows of $A$. Using the notations above, we have the equation
 \begin{equation}\label{eqn:Lap2E3}
 	\frac{Q_1}{a_{1,1}}+ \frac{Q_2}{a_{1,2}} + \frac{Q_3}{a_{1,3}}  = Q_0. 
 \end{equation}

 As in Section~\ref{laplace}, we bound the cardinality of the following sets for $k=0,\ldots,3$,
 \begin{align*}
 	\widetilde \cD_3&(\Z^{-1};H,k,\delta)\\
 	& =\{A\in \cD_3\(\Z^{-1};H,\delta\):~ \text{exactly } k \text{ of } Q_1,Q_2,Q_3\text{ are nonzero}\}. 
 \end{align*}
 
By permuting the rows of $A$, we can assume that analogues of  the equation~\eqref{eqn:Lap2E3}
 obtained with respect to other rows all have at most $k$ nonzero coefficients. In particular, in the case $k=0$, 
all  $2\times 2$ submatrices of $A$ are singular.
 
We first prove that the case $k=1$ does not exist. In other words,
we  prove that the set $\widetilde\cD_3(\Z^{-1};H,1,\delta)$ is empty for all $\delta$.

 	Let 
\[A=\begin{pmatrix}
 		1/a_{1,1}&1/a_{1,2}&1/a_{1,3}\\
 		1/a_{2,1}&1/a_{2,2}&1/a_{2,3}\\
 		1/a_{3,1}&1/a_{3,2}&1/a_{3,3}
 	\end{pmatrix} \in \widetilde\cD_3(\Z^{-1};H,1,\delta).
\] 
From its definition, exactly two of $Q_1$, $Q_2$, and $Q_3$ is zero. Without loss of generality, suppose that $Q_2=Q_3=0$. This means
\[ 
 \det \begin{pmatrix}
1/a_{2,1}&1/a_{2,2}\\
1/a_{3,1}&1/a_{3,2}
 	\end{pmatrix} = \det \begin{pmatrix}
 	1/a_{2,1}&1/a_{2,3}\\
 	1/a_{3,1}&1/a_{3,3}
 	\end{pmatrix}=0.\]
 	
 	By the determinant formula, we have 
\[
 		a_{2,1}a_{3,2}=a_{2,2}a_{3,1} \mand a_{2,1}a_{3,3}=a_{3,1}a_{2,3}.
\]
 By elimination, since all variables are nonzero, we have 
 \[
 	a_{2,2}a_{3,3}=a_{3,2}a_{2,3}.
\] 
This implies $Q_1=0$, which, however,  contradicts the definition of $\widetilde\cD_3(\Z^{-1};H,1,\delta)$. This completes the proof.
 
 Now we just need to consider the cases $k=0,2,3$. For the case $k=3$, we fix the last two rows of $A$ in $O(H^6)$ ways. Next, we may apply Lemma~\ref{lem:aixi-Z n=2,3} to the equation~\eqref{eqn:Lap2E3} to see that there are at most $H^{1+o(1)}$ possible values for $(1/a_{1,1},1/a_{1,2},1/a_{2,3})$. Multiplying these, we have \begin{equation}\label{eq:Z-1k3}
\widetilde\cD_3(\Z^{-1};H,3,\delta) \le H^{7+o(1)}
 \end{equation} for all $\delta$.
 
 For the case $k=2$, without loss of generality let $Q_3=0$. From the case $n=2$ and $\delta=0$, we may choose the elements of the submatrix 
 \[\begin{pmatrix}
 	1/a_{2,1}&1/a_{2,2}\\
 	1/a_{3,1}&1/a_{3,2}
	 \end{pmatrix}
\] 
in at most $H^{2+o(1)}$ ways. We may also choose the elements of the last column of $A$ in $H^3$ ways. To bound the number of possible tuples $(1/a_{1,1},1/a_{1,2})$, we se that in this case, the equation~\eqref{eqn:Lap2E3} becomes 
 \[
 Q_1/a_{1,1}+Q_2/a_{1,2}=Q_0.
\]
Using Lemma~\ref{lem:aixi-Z n=2,3}, we see that this equation has at most $H^{o(1)}$ solutions if $\delta\ne 0$. Furthermore, if $\delta=0$, we can see that this equation has at most $O(H)$ solutions. Multiplying all bounds, we obtain 
\begin{equation}\label{eq:Z-1k2}
\widetilde \cD_3(\Z^{-1};H,2,\delta) \le \begin{cases}
 	H^{5+o(1)}&,\text{ if }\delta \ne 0,\\
 	H^{6+o(1)}&,\text{ if }\delta = 0.
 \end{cases}
 \end{equation}

 For the case $k=0$ (which can only happen when $\delta=0$), as we have mentioned above, 
  we can assume that every $2\times 2$ submatrix of $A$ is singular. 
 We recall that there are at most $H^{2+o(1)}$ ways to choose the matrix 
 \[
 \begin{pmatrix}
 	1/a_{2,1}&1/a_{2,2}\\
 	1/a_{3,1}&1/a_{3,2}
 \end{pmatrix}
 \] that is singular. We fix such a choice. Next, we consider the matrix \[\begin{pmatrix}
 1/a_{1,1}&1/a_{1,2}\\
 1/a_{2,1}&1/a_{2,2}
 \end{pmatrix},\] which  is singular.  Therefore,  
\[
 a_{1,1}a_{2,2}=a_{1,2}a_{2,1}.
\]
We may now choose $a_{1,1}$ in $O(H)$ ways.
Since $a_{2,1}$ and $a_{2,2}$ are already fixed, $a_{1,2}$ is uniquely defined.
Thus we see that there are at most $H^{3+o(1)}$ choices for the above six entries 
of $A$ and the entries which are still free are 
$\(a_{1,3}, a_{2,3}, a_{3.3}\)$.  We can estimate their quantity non-trivially but even the trivial
bound $O(H^3)$ already implies  
\begin{equation}\label{eq:Z-1k0}
 	\widetilde\cD_3(\Z^{-1};H,0,\delta) \le H^{6+o(1)},
 \end{equation}
 which is satisfactory.

 Combining equations~\eqref{eq:Z-1k3}, \eqref{eq:Z-1k2} and~\eqref{eq:Z-1k0} completes the proof of the case $n=3$.

\subsection{The case $n\ge 4$} 
We first consider  the problem of bounding $\#\cD_n(\Z^{-1};H,0)$,
 the number of $n\times n$ matrices with entries in $\cE(H)$ and of determinant 0. We note that all matrices in the corresponding set have rank at most $n-1$. Therefore, we have 
\[
    \#\cD_n(\Z^{-1};H,0) \le \sum_{r=1}^{n-1} \#\cL_{n,n,r} (\Z^{-1};H),
\]
where $\#\cL_{n,n,r} (\Z^{-1};H)$ is the number of matrices in $\cM_n(\Z^{-1};H)$ with rank $r$. Recalling Theorem~\ref{thm:rank 1/Z}, we have 
\begin{align*}
\#\cL_{n,n, r}&\(\Z^{-1};H\) \\
&\le \max\left\{H^{nr+n-r} ,  H^{nr+2n-2r+o(1) }  ,  H^{(n-r)(r+1)/2+nr+o(1)} \right\}\\
&\le \begin{cases}
    H^{n+o(1)}, & \text{ if }  r=1,\\
    H^{3n-2+o(1)}, & \text{ if }  r=2,\\
    H^{(3nr+n-r-r^2)/2+o(1)},& \text{ if } r\ge 3.
\end{cases}
\end{align*} 
Therefore, we have
\begin{align*}
    \#\cD_n(\Z^{-1};H,0) &\le \sum_{r=1}^{n-1} \#\cL_{n,n, r}\(\Z^{-1};H\), \\
    &\le H^{n+o(1)}+H^{3n-2+o(1)}+ \sum_{r=3}^{n-1} H^{(3nr+n-r-r^2)/2+o(1)},\\
    &\le H^{3n-2+o(1)}+ H^{n^2-n/2+o(1)} \le H^{n^2-n/2+o(1)},
\end{align*}  
where  we have used that  $(3nr+n-r-r^2)/2$ is increasing in $r\in [3,n-1]$. This completes the proof in the case of $\#\cD_n(\Z^{-1};H,0)$.    
 
 Next, we move to  the case $\delta \ne 0$. 	Consider an $n\times n$ matrix $A=\(a_{i,j}\)_{i,j=1}^n$ with entries in $\cE(H)$ and determinant $\delta$, and fix the last $n-1$ rows of $A$. The proof follows exactly as the proof in Section~\ref{laplace}. Indeed, following same discussion line by line, and notations, therein (with $\Q$ replaced by $\Z^{-1}$), we arrive at the equation
 	\begin{equation}\label{eqn:Lap2E}
 		\sum_{j=1}^n Q_j a_{1,j} = Q_0, \qquad a_{1,j}\in\cE(H),\quad j=1,\ldots,n,
 	\end{equation}
and we can write
 	\[	
 	\cD_n\(\Z^{-1};H,\delta\) =\bigcup_{k=1}^n	\widetilde \cD_n(\Z^{-1};H,k,\delta),
 	\]
	where $Q_j$, $j=0,\ldots,n$, and $\cD_n(\Z^{-1};H,k,\delta)$, $k=1,\ldots,n$, are defined as in Section~\ref{laplace}.

 	To estimate $\# \widetilde \cD_n(\Z^{-1};H,k,\delta)$, 
 	without loss of generality we can only consider $A \in \widetilde \cD_n(\Z^{-1};H,k,\delta)$ 
 	with  $Q_0, Q_1, \ldots, Q_k\ne 0$. 
 	
 The rest of the proof also follows same line as the proof of the previous section, but for completeness we provide the details such that the bounds we obtain are clear.
 
 	We note that the equation~\eqref{eqn:Lap2E} satisfies the conditions of Lemma~\ref{lem:aixi-Z}. 
 	
 	We first consider the case $k = n$ and bound $\# \widetilde \cD_n(\Z^{-1};H,n,\delta)$. In this case,  we can directly apply Lemma~\ref{lem:aixi-Z-inhom} to conclude  that the equation~\eqref{eqn:Lap2E}   has at most $H^{n/2-1/(2n-2)+o(1)}$ solutions. 
 	This implies that for each  possible  $(n-1)\times n$ bottom  sub-matrix of $A$, there are at most $H^{n/2-1/(2n-2)+o(1)}$ possible choices for the first row of $A$. Since we can choose the bottom $n-1$ rows in  
 	$O\(H^{(n-1)n}\)$ ways, we see that there are at most
 	\begin{equation}
 		\label{eq:k=nE}
		\begin{split}
 \widetilde \cD_n(\Z^{-1};H,n,\delta)&  \le  
	H^{n/2-1/(2n-2)+o(1)} \cdot H^{(n-1)n}\\
	& \le  	H^{n^2-n/2-1/(2n-2)+o(1)}
\end{split}
\end{equation}
 	possible matrices $A$ with $\det A = \delta$ in this case.

  Now we give an upper bound for $\# \widetilde \cD_n(\Z^{-1};H,k,\delta)$ for all $1\le k <n$. Let $A\in  \widetilde \cD_n(\Z^{-1};H,k,\delta)$ for some $k< n$.  In particular, we have $Q_n=0$ and $Q_1\ne 0$. The first equality implies $\det M_{1,n}=0$. Hence, there are  $\# \cD_{n-1}\(\Z^{-1};H,0\)$ possibilities for the $(n-1)\times (n-1)$ submatrix $M_{1,n}$ of $A$ that has determinant zero. Moreover, there are at most $H^{n}$ ways of choosing the last column $(a_{1,n},\ldots,a_{n,n})^T$. Hence, in total, there are at most $H^{n}\#\cD_{n-1}\(\Z^{-1};H,0\)$ possibilities for $a_{1,n}$ and the bottom $n-1$ rows of $A$. 
  We now fix $a_{1,n}$ and the bottom $n-1$ rows, and we need to count only the number of possibilities for  $(a_{1,1},\ldots,a_{1,n-1})$.
 	
 	To count the number of possible tuples $(a_{1,1},\ldots,a_{1,n-1})$, we look at the equation~\eqref{eqn:Lap2E}, recalling that we assume that $Q_1,\ldots,Q_k\ne 0$ for $1\le k\le n-1$. If $k=1$, then $a_{1,1}$ is uniquely defined and we can choose $(a_{1,2},\ldots,a_{1,n-1})$ in $H^{n-2}$ ways. Hence, in this case the number of possible matrices $A$ is 
 	\begin{equation}
 		\label{eq:k=1E}
				\begin{split}
 \widetilde \cD_n(\Z^{-1};H,1,\delta) & \ll 
 		H^{n}\#\cD_{n-1}\(\Z^{-1};H,0\) \cdot H^{n-2}\\
		& =H^{2n-2}\#\cD_{n-1}\(\Z^{-1};H,0\).
\end{split}
 	\end{equation}
 	
Next, for $k=2$ and $k=3$, we  apply Lemma~\ref{lem:aixi-Z n=2,3} to the equation~\eqref{eqn:Lap2},
	 from where we get that there are at most $H^{o(1)}$ possible values of $\(a_{1,1},a_{1,2}\)$ for $k=2$ and $H^{1+o(1)}$ possible values of $\(a_{1,1},a_{1,2},a_{1,3}\)$ for $k=3$, respectively. Using the trivial bound $H^{n-3}$ (for $k=2$) and $H^{n-4}$ (for $k=3$), respectively, for the rest of entries, the number of possible matrices $A$ for $k=2$ and $k=3$ is bounded above as following:  

 	\begin{equation}
	\begin{split}
 		\label{eq:k=2E}
		\widetilde \cD_n(\Z^{-1};H,k,\delta)& \le
 		H^{n}\#\cD_{n-1}\(\Z^{-1};H,0\) \cdot H^{n-3+o(1)}\\
		& =  H^{2n-3+o(1)}\#\cD_{n-1}\(\Z^{-1};H,0\),
		\end{split} 
 	\end{equation} 

We assume now that $4\le k<n$ and apply Lemma~\ref{lem:aixi-Z-inhom} to the equation~\eqref{eqn:Lap2E}, from where we get that there are at most $H^{k/2-1/(2k-2)+o(1)}$ possible values of $\(a_{1,1},a_{1,2},\ldots, a_{1,k}\)$. Using the trivial bound $H^{n-1-k}$ for the rest of entries $a_{1(k+1)},\ldots, a_{1,n-1}$, the number of possible matrices $A$ in this case is
 	\begin{equation}
	\begin{split}
 		\label{eq:k>2E}
		&  \widetilde \cD_n(\Z^{-1};H,k,\delta)\\
		 &\qquad \qquad  \le 
 		H^{n}\#\cD_{n-1}\(\Z^{-1};H,0\) \cdot H^{k/2-1/(2k-2)+o(1)}\cdot H^{n-1-k}\\
		  &\qquad \qquad  \le H^{2n-k/2 -1 -1/(2k-2)+o(1)}\#\cD_{n-1}\(\Z^{-1};H,0\)\\
		 &\qquad \qquad  \le H^{2n-19/6+o(1)}\#\cD_{n-1}\(\Z^{-1};H,0\),
		\end{split} 
 	\end{equation} where the last inequality is true since $k/2+1+1/(2k-2)$ attain its minimum at $k=4$.
	 	
 	Putting the bounds~\eqref{eq:k=nE}, \eqref{eq:k=1E},  \eqref{eq:k=2E}  and~\eqref{eq:k>2E}  together, we have 
 \begin{align*}
&\#\cD_n\(\Z^{-1};H,\delta\)\\
& \qquad \quad \le \max\left\{H^{n^2-n/2-1/(2n-2)+o(1)}, H^{2n-2}\#\cD_{n-1}\(\Z^{-1};H,0\) \right\}.
 \end{align*}
Applying our results for $\#\cD_{n-1}\(\Z^{-1};H,0\)$ completes the proof.

\section{Proof of Theorem~\ref{thm:pnhfE}} 
\subsection{The case $n=2$ and $f(X)\neq X^2$} We first bound $ 	\#\cP_{n}(\Z^{-1};H,f)$ for the case $n=2$. Suppose that 
\[
 f(X)=X^2-\alpha X+\beta,
\] 
for some $\alpha, \beta\in \Q$.
 
 We recall that for a matrix
 \[A=\begin{pmatrix}
 	1/x_1&1/x_2\\1/x_3&1/x_4
 \end{pmatrix}\in  \cP_{2}(\Z^{-1};H,f),
 \] 
 its determinant $\det A$ is $\beta$. Now, if $\beta \neq 0$, this implies  
\[
 \#\cP_{n}(\Z^{-1};H,f)\le \#\cD_{n}(\Z^{-1};H,\beta) \le H^{o(1)},
\] 
which proves the assertion.
 
 Now, suppose $\beta=0$. This implies that 
 \begin{equation}\label{eq:x1x4x2x3}
 	x_1x_4=x_2x_3.
 \end{equation} We recall that  	$\alpha=\Tr(A)$. Let $\alpha=r/s$ with $\gcd(r,s)=1$. We note that $r\ne 0$. We have  \begin{equation}\label{eq:x1x4}
\frac{r}{s}=\frac{1}{x_1}+\frac{1}{x_4},
 \end{equation} which is equivalent to
 \[
(rx_1+s)(rx_4+s)=s^2.
\] 
By applying the same argument as in Section~\ref{sec:n=2E}, there are at most 
 $H^{o(1)}$ possible solutions of the equation~\eqref{eq:x1x4}. Then, coming to the equation~\eqref{eq:x1x4x2x3}, we see that, by Lemma~\ref{lem:div} for each choice of $\(x_1,x_4\)$ there are at most $|x_1x_4|^{o(1)}\le H^{o(1)}$ possible choices 
 of $\(x_2,x_3\)$. We conclude that if $\beta=0$ and $\alpha\ne 0$, there are  at most $H^{o(1)}$ matrices $A\in \cP_{2}(\Z^{-1};H,f)$.  This completes the proof of the case $n=2$ and $f(X)\neq X^2$.
 
 \subsection{The case $n=2$ and $f(X)= X^2$} Let \[A=\begin{pmatrix}
 	1/x_1&1/x_2\\1/x_3&1/x_4
 \end{pmatrix}\in  \cP_{2}(\Z^{-1};H,X^2).
 \] Using algebraic manipulation and the fact that $\Tr A=\det A=0$, we have that $A\in   \cP_{2}(\Z^{-1};H,X^2)$ if and only if the entries of $A$ satisfy \begin{equation}\label{eq:X^2}
 	x_1x_4=x_2x_3 \mand  x_1=-x_4,
 \end{equation} 
 and $|x_1|,|x_2|,|x_3|,|x_4|\le H$. Therefore, we have 
 \[
 x_1^2=-x_2x_3.
\]
 
 From the last equation, we see that counting $\# \cP_{2}(\Z^{-1};H,X^2)$ is equivalent to (up to sign changes)  counting the number of integer pairs $(a,b)$ with $1\le a,b \le H$ such that $ab$ is a perfect square. More precisely, we note that each possible integer pair $(a,b)$ with $1\le a,b \le H$ such that 
 \[
 ab=c^2
\] 
for some integer $1\le c\le H$ corresponds to four quadruplets of solutions of the equation~\eqref{eq:X^2}, with the following bijections: \begin{align*}
 	(a,b)&\mapsto (c,a,-b,-c)\\
 	(a,b)&\mapsto (c,-a,b,-c)\\
 	(a,b)&\mapsto (-c,a,-b,c)\\
 	(a,b)&\mapsto (-c,-a,b,c).
 \end{align*}
We then apply Lemma~\ref{lem:square} to finish the proof.
 
 \subsection{The case $n\ge 3$} 
 We proceed as in the proof of Theorem~\ref{thm:pnhf}. Indeed, we count matrices $A \in \cP_n(\Z^{-1};f)$, and thus, in particular,  matrices $A=(1/a_{i,j})_{1\le i,j\le n}$, $1/a_{i,j}\in\cE(H)$, for which $\Tr A$ and $\Tr A^2$ are fixed. These lead again, after clearing denominators, to two equations in Egyptian fractions:
\begin{equation}
\label{eq:treq egyp}
s_1\frac{1}{a_{1,1}}+\cdots+s_1\frac{1}{a_{n,n}}=r_1
\end{equation}
in variables $1/a_{i,i}\in\cE(H)$, $i=1,\ldots,n$, with fixed $r_1,s_1\in\Z$, $s_1\ne 0$, of size
\[
|r_1|,|s_1|\le H^{O(1)},
\]
and
\begin{equation}
\label{eq:tr2eq egyp}
\sum_{1\le i<j \le n}s_2\frac{1}{A_{i,j}}=r_2,\qquad A_{i,j}=a_{i,j}a_{j,i},
\end{equation}
in variables $1/A_{i,j}\in\cE(H^2)$, $1\le i<j\le n$, with fixed $r_2\in\Z$, $s_2\ne 0$, of size
\[
|r_2|,|s_2|\le H^{O(1)}.
\]

If $n=3$, then by Lemma~\ref{lem:aixi-Z n=2,3}, the number of solutions to the equation~\eqref{eq:treq egyp} is at most 
$H^{1+o(1)}$. 
The equation~\eqref{eq:tr2eq egyp} has $3(3-1)/2=3$ variables $A_{i,j}$, $1\le i<j\le 3$, thus we can apply again Lemma~\ref{lem:aixi-Z n=2,3} to obtain at most $H^{2+o(1)}$ solutions $1/A_{i,j}\in\cE(H^2)$, $1\le i<j\le n$. Since $a_{i,j}, a_{j,i}\mid A_{i,j}$ for all $1\le i<j\le n$, by Lemma~\ref{lem:div}, for each fixed $A_{i,j}$, we have at most $H^{o(1)}$ possibilities for $1/a_{i,j},1/a_{j,i}\in\cE(H)$.

Therefore, in total, we have at most
\[
H^{3+o(1)}
\]
matrices $A\in  \cP_3(\Z^{-1};H,f)$.

We consider now the case $n\ge 4$, and the following two subcases based on the coefficient $f_{n-1}$ of $X^{n-1}$ of $f$ being zero or not.

If $f_{n-1}\ne 0$, then we can apply Lemma~\ref{lem:aixi-Z-inhom} to~\eqref{eq:treq egyp} and obtain at most
\[
H^{n/2-1/(2n-2)+o(1)}
\]
solutions $1/a_{i,i}\in\cE(H)$, $i=1,\ldots,n$.

Applying Lemma~\ref{lem:aixi-Z}  to~\eqref{eq:tr2eq egyp}, we obtain  
\[
H^{2\cdot(n^2-n)/4+o(1)}=H^{(n^2-n)/2+o(1)}
\]
 solutions $A_{i,j}\in\cE(H^2)$, $1\le i<j\le n$. As above, since $a_{i,j}, a_{j,i}\mid A_{i,j}$ for all $1\le i<j\le n$, by Lemma~\ref{lem:div}, for each fixed $A_{i,j}$, we have at most $H^{o(1)}$ possibilities for $1/a_{i,j},1/a_{j,i}\in\cE(H)$. 

Therefore, we conclude that there are at most 
\[
H^{n/2-1/(2n-2)+(n^2-n)/2+o(1)}=H^{n^2/2-1/(2n-2)+o(1)}
\]
matrices $A\in  \cP_n(\Z^{-1};f)$ when $f_{n-1}\ne 0$.

If $f_{n-1}=0$, then we apply Lemma~\ref{lem:aixi-Z}  to both~\eqref{eq:treq egyp} and~\eqref{eq:tr2eq egyp}, to conclude as above that we have  at most
\[
H^{n/2+(n^2-n)/2+o(1)}=H^{n^2/2+o(1)}
\]
matrices $A\in  \cP_n(\Z^{-1};f)$ when $f_{n-1}= 0$. This concludes the proof.

 \section{Comments} 
 First we note that the questions we considered in this paper can also be asked in the setting of symmetric matrices, and other special matrices, which certainly requires new ideas.
 
 It is also easy to see that one of possible ways to improve 
 Theorem~\ref{thm:pnhfE} is to get a good bound on the number of solutions of the system of equations 
 \[
    \sum_{i=1}^n Q_i/x_i = Q_0 \mand 
       \sum_{i=1}^n R_i/x_i^2 = R_0 
\]
with non-zero coefficients $Q_i$ and $R_i$ 
and with $1/x_i \in \cE(H)$, $i=1,\ldots,n$.  
Indeed, having  a bound 
we can  control  the frequency of the solutions to~\eqref{eq:treq egyp}, which lead to the value $r_2 = 0$ in~\eqref{eq:tr2eq egyp}, enabling us to 
apply  Lemma~\ref{lem:aixi-Z-inhom} to the 
equation~\eqref{eq:tr2eq egyp} for each the remaining solutions to~\eqref{eq:treq egyp}.

In fact, this seems to be a question of independent interest, and in fact 
 also suggests to consider more general systems with $s$ consecutive 
 powers:
 \[
    \sum_{i=1}^n Q_{i,\nu}/x_i^\nu = Q_{0,\nu}, \qquad 
 \nu =1, \ldots, s. 
\]
Even the case of $Q_{i,\nu} =\pm 1$, $Q_{0,\nu} = 0$, 
$i=1, \ldots, n$, $\nu =1, \ldots, s$, is already interesting. 

One can also study similar systems of equation in Farey fractions, 
which also  seems to be a new interesting question. 

Finally, it is quite natural to attempt to combine the ideas of Blomer and Li~\cite{BlLi} with bounds of Theorem~\ref{thm:rank Q} and~\ref{thm:rank 1/Z} and study correlations between the values of linear forms in 
Farey and Egyptian fractions. 
This may require quite significant efforts, though.

 \section*{Acknowledgements} 
 The authors are grateful to the referee for carefully reading our paper and for useful comments and also to Nick Rome
 for noticing that in the journal version the main and  error terms in Theorem~\ref{thm:pnhfE} and Lemma~\ref{lem:square} 
 all contain an extra factor of $\log H$. 
 
 During the preparation of this paper, M.~A., A.~O. and I.~S.  were supported by  Australian Research Council Grants DP200100355 and  DP230100530, M.~A. was also supported by a UNSW Tuition Fee Scholarship, and V.~K. was supported by NSF Mathematical Sciences Research program grant DMS-2202128. 
V.~K., A.~O., I.~S. would also like to thank the Mittag-Leffler Institute for its hospitality 
and excellent working environment.

\end{document}